\documentclass[11pt,dvipsnames,reqno]{amsart}
\usepackage[T1]{fontenc}
\usepackage{a4wide}
\usepackage{yhmath,mathrsfs,amsthm,amsmath,amssymb,amsfonts,enumerate,lipsum, appendix,mathtools, bm}
\usepackage{bbold,comment}
\usepackage{pdfpages}
\usepackage{orcidlink}
\usepackage[mathscr]{euscript}
\usepackage{graphicx}
\usepackage{pgf,tikz}
\usepackage{tkz-euclide}
\usepackage{graphicx}
\usepackage{subcaption}
\usetikzlibrary{shapes.geometric}
\usetikzlibrary{arrows,hobby}
\usepackage[shortlabels]{enumitem}
\usepackage[english]{babel}

\allowdisplaybreaks
\newcommand{\tb}{\textcolor{blue}}

\usepackage[sort,numbers]{natbib}
\usepackage{hyperref}
\usepackage{esint}
\hypersetup{
	colorlinks=true,
	linkcolor=b,
	filecolor=b,      
	urlcolor=b,
	citecolor=ao,
}
\definecolor{b}{rgb}{0.0, 0.0, 1.0}
\definecolor{ao}{rgb}{1.0, 0.13, 0.32}
\usepackage[margin=1in, heightrounded]{geometry}
\usepackage{bigints}
\usepackage{esint}
\newtheorem{theorem}{Theorem}[section]
\newtheorem{lemma}[theorem]{Lemma}
\newtheorem{proposition}[theorem]{Proposition}

\newtheorem{definition}[theorem]{Definition}
\theoremstyle{definition}

\newtheorem{remark}[theorem]{Remark}
\numberwithin{equation}{section}
\usepackage[hyperpageref]{backref}

\newcommand*\rd{\mathbb{R}^d}

\newcommand{\al} {\alpha}

\newcommand{\pa} {\partial}
\newcommand{\be} {\beta}
\newcommand{\de} {\delta}

\newcommand{\Om} {\Omega}
\newcommand{\la} {\lambda}
\newcommand{\La} {\Lambda}

\newcommand{\Gr} {\nabla}
\newcommand{\no} {\nonumber}
\newcommand{\noi} {\noindent}
\newcommand{\eps} {\varepsilon}

\newcommand{\ra} {\rightarrow}

\newcommand{\bee} {\begin{equation}}
	\newcommand{\eee} {\end{equation}}
\newcommand{\bea} {\begin{eqnarray}}
	\newcommand{\eea} {\end{eqnarray}}
\newcommand{\Bea} {\begin{eqnarray*}}
	\newcommand{\Eea} {\end{eqnarray*}}

\def\d{\,{\rm d}}
\def\dx{\,{\rm d}x}
\def\dy{\,{\rm d}y}
\def\dt{\,{\rm d}t}

\def\ds{\,{\rm d}s}

\def\RN{{\mathbb R}^N}
\def\({{\Big(}}
\def\){{\Big)}}

\def\dt{\,{\rm d}t}
\def\dx{\,{\rm d}x}

\DeclarePairedDelimiter\abs{\lvert}{\rvert}%
\DeclarePairedDelimiter\norm{\lVert}{\rVert}%

\vspace{-1.5\baselineskip}

\let\tmp\phi \let\phi\varphi \let\varphi\tmp

\newcommand{\frap}{(-\Delta_p)^s}
\newcommand{\delp}{-\Delta_p}
\newcommand{\pst}{p^\ast}
\newcommand{\psst}{p_s^\ast}
\newcommand{\R}{\mathbb{R}}
\newcommand{\RR}{\mathbb{R}^N}

\newcommand{\ov}{\overline}

\newcommand{\NN}{\mathbb{N}}
\newcommand{\ZZ}{\mathbb{Z}}
\renewcommand{\AA}{\mathcal{A}}

\newcommand{\PP}{\mathcal{P}}
\renewcommand{\SS}{\mathcal{S}}

\newcommand\restr[2]{{
  \left.\kern-\nulldelimiterspace 
  #1 
  \littletaller 
  \right|_{#2} 
  }}
\usepackage{natbib}
\newcommand{\littletaller}{\mathchoice{\vphantom{\big|}}{}{}{}}

\title[Quasilinear mixed local-nonlocal concave-convex]{Quasilinear problems with mixed local-nonlocal operator and concave-critical nonlinearities: Multiplicity of positive solutions}
\author[M. Bhakta, N. Biswas, and P. Das]{Mousomi Bhakta\,$^\dagger$, Nirjan Biswas\,\orcidlink{0000-0002-3528-8388}, \and Paramananda Das\,\orcidlink{0009-0009-2821-6675}}
\address{Department of Mathematics, Indian Institute of Science Education and Research (IISER-Pune), Dr. Homi Bhabha Road, Pune-411008, India}
\email[M. Bhakta]{mousomi@iiserpune.ac.in} \email[N. Biswas]{nirjan.biswas@acads.iiserpune.ac.in, nirjaniitm@gmail.com} 
\email[P. Das]{paramananda.das@students.iiserpune.ac.in, pd348225@gmail.com}
\thanks{$^\dagger$Corresponding author}

\begin{document}

\begin{abstract}
 We study the existence and multiplicity of positive solutions for the following concave-critical problem driven by an operator of mixed order obtained by the sum of the classical $p$-Laplacian and of the fractional $p$-Laplacian, 
\begin{equation}\tag{$\PP_{\lambda,\varepsilon}$}
    -\Delta_p u+\varepsilon(-\Delta_p)^s u=\lambda|u|^{q-2}u+|u|^{p^*-2}u \;\text{ in }\Omega,\quad
    u=0 \; \text{ in }\mathbb{R}^N \setminus \Omega,
\end{equation}
where $\Omega\subset\mathbb{R}^N$ is a bounded open set, $\varepsilon\in(0,1]$, $0<s<1<q<p<N$, and
$p^*=\frac{Np}{N-p}$, and $\lambda \in \mathbb{R}$ is a parameter. For $\lambda \leq 0$, we show that (\textcolor{blue}{$\PP_{\lambda,\varepsilon}$}) has no nontrivial solution on star-shaped domains. For $\lambda>0$, we prove Ambrosetti-Brezis-Cerami type results. In particular, we prove the existence of $\Lambda_\varepsilon$ such that (\textcolor{blue}{$\PP_{\lambda,\varepsilon}$}) has a positive minimal solution for $0<\lambda<\Lambda_\varepsilon$, a positive solution for $\lambda=\Lambda_\varepsilon$ and no positive solution for $\lambda>\Lambda_\varepsilon$. We also prove the existence of $0<\lambda^\#\leq\Lambda_\varepsilon$ such that (\textcolor{blue}{$\PP_{\lambda,\varepsilon}$}) has at least two positive solutions for $\lambda\in(0,\lambda^\#)$ provided $\varepsilon$ small enough. This extends the recent result of Biagi and Vecchi (Nonlinear Anal. 256 (2025), 113795), Amundsen, et al. (Commun. Pure Appl. Anal., 22(10):3139–3164, 2023) from $p=2$ to the general $1<p<N$. Additionally, it extends the classical result of Azorero and Peral (Indiana Univ. Math. J., 43(3):941–957, 1994) to the mixed local-nonlocal quasilinear problems. Moreover, our results complement the multiplicity results for nonnegative solutions in da Silva, et al. (J. Differential Equations, 408:494–536, 2024).
\end{abstract}
\maketitle
\section{Introduction}
 In this paper we consider the following  mixed local-nonlocal concave-critical problem:
\begin{equation}\tag{$\PP_{\lambda,\varepsilon}$}\label{main_PDE}
\begin{cases}
            \delp u+\eps\frap u=\lambda|u|^{q-2}u+|u|^{\pst-2}u&\text{in }\Om,\\
            u=0&\text{in }\RR \setminus \Om,
        \end{cases}
\end{equation}
where $\eps\in(0,1]$, $0<s<1<q<p<N$ and $\Omega\subset\RR$ is a bounded domain of class $C^{1,\alpha}$ for some $\alpha\in(0,1)$, $\pst=\frac{Np}{N-p}$ is the critical Sobolev exponent, $\lambda \in \R$ is a parameter, and $p$-Laplace, fractional $p$-Laplace operators are defined for smooth enough functions as
\begin{align*}
    &\delp u = -\text{div}(\abs{ \Gr u}^{p-2} \Gr u); \\
    &\frap u = \, \textrm{P.V.}\int _{\RR} \frac{|u(x)-u(y)|^{p-2}(u(x)-u(y))}{|x-y|^{N+ps}} \dy,
\end{align*}
where P.V. is the Cauchy principal value.

The study of concave-convex problems when $\eps=0$ is quite old and dates back
to \cite{ABC, GaPa1}. Recently, the nonlinear problems driven by operators of mixed type have gained a lot of interest and
intensive investigation, in connection with the study of optimal animal foraging strategies (see e.g. \cite{DiVal}). On the other hand, when $p=2$, the operator $-\Delta_p+\eps(-\Delta_p)^s$ can
be seen as the infinitesimal generator of a stochastic process having both a Brownian motion
and a L\'evy f{}light. Hence, there is a vast literature which establishes several regularity
properties adopting probabilistic techniques, see e.g. \cite{CKSV} and the references therein. Mathematically speaking, this operator offers quite relevant challenges caused by the combination of nonlocal difficulties with the lack of invariance
under scaling. In the case $p=2$, Biagi, et al. \cite{BDVV} studied the Brezis-Nirenberg problem, and very recently da Silva, et al. \cite{DFB} generalized the Brezis-Nirenberg problem to any $p \in (1, \infty)$. The regularity results for mixed local-nonlocal operators are studied in \cite{AC, BMS, DM, PG, GK,SVWZ} and the references there-in. da Silva and Salort \cite{DSS} studied the existence of a positive solution of $(\PP_{\la,1})$ and its asymptotic behaviour when $p\to\infty$.


In the seminal paper \cite{ABC}, Ambrosetti-Brezis-Cerami (ABC) established the existence of $\La>0$ such that the problem
$$-\Delta u=\lambda u^{q-1}+u^{r-1} \text{ in }\Om,\quad u>0\text{ in }\Om,\quad u=0\text{ on }\pa\Om,$$
with $1<q<2<r$, admits
\begin{itemize}
    \item a  minimal solution for any $\la<\La$,
    \item at least one solution when $\la=\La$,
    \item no solution when $\la>\La$, and
    \item a second solution for any $\la\in(0,\La)$, provided $r\leq 2^*$.
\end{itemize}
In the literature, numerous authors investigated (ABC) type results for various elliptic operators. Boccardo, et al.  in \cite{BEP} extended some of the above results for $p$-Laplacian. More precisely,  they obtained the existence of a positive minimal solution of
\begin{align}\label{BEP-p}
    \delp u=|u|^{r-2}u+\la g(u) \text{ in }\Om,\quad u>0\text{ in }\Om,\quad u=0\text{ on }\pa\Om,
\end{align}
where  $g(s)\leq c_1 s^{q-1}$ for all $s\geq0$ and $u\mapsto |u|^{r-2}u+\la g(u)$ is nondecreasing. They found a $\la_1>0$ such that \eqref{BEP-p} has no positive solution when $\la>\la_1$. In the purely nonlocal setup, \cite{BCSS, YZ} proved the (ABC) type result. Amundsen, et al. in \cite{AMT}, showed the existence of $\La$ such that the following problem
\begin{equation}
    -\Delta u+(-\Delta)^s u=\lambda|u|^{q-2}u+|u|^{p-2}u\text{ in }\Om,\quad u=0\text{ in }\RR \setminus \Om,
\end{equation}
with $1<q<2<p$, has a positive solution when $\la<\La$, and has no positive solution when $\la>\La$. Very recently in Dhanya et al. \cite{DGJ} considered the problem
\begin{equation}\label{dgj_eq}
    \delp u+(-\Delta_q)^s u=\la\left(a(x)|u|^{\de-2}u+b(x)|u|^{r-2}u\right) \text{ in }\, \Om,\quad u=0\, \text{ in }\,\Om^c,
\end{equation}
where $\Om$ is a bounded domain in $\RN$, $p<q$ or $q<p$, $1<\de<\min\{p,q\}$, and  $r\leq \max\{p^*,q_s^*\}$, where $q_s^*=\frac{Nq}{N-sq}$. When $p<q$ and $r<\max\{p^*,q_s^*\}$ or $q<p$ and $r=p^*$ (the second case with $b=1$ or $b=\la^{-1}$), they studied multiplicity of nonnegative nontrivial solutions using the fibering map analysis and constrained minimization on specific subsets of the Nehari manifold.

In this paper, we extend the results of Amundsen, et al \cite{AMT} and Biagi-Vecchi \cite{BV2} to the general $p \in (1, \infty)$. We also extend the results of Azorero and Peral \cite{GaPa1} to the case of mixed local-nonlocal quasilinear operators. Moreover, our results complement the multiplicity results for nonnegative solutions obtained by da Silva, et al in \cite{DFB}.

For $\la \le 0$, we prove the non-existence of a nontrivial solution for \eqref{main_PDE}. In the celebrated paper \cite{BN1}, using the Pohozaev identity, Brezis-Nirenberg  showed that
\begin{equation*}
    -\Delta u=\lambda u+u^{\pst-1}\text{ in }\Om,\quad u>0 \text{ in }\Om,\quad u=0\text{ in }\pa\Om,
\end{equation*}
has no solution when $\Om$ is a bounded and star-shaped domain. For $p=2$, a similar non-existence result can be obtained in the purely nonlocal setup and mixed local-nonlocal setup by using the Pohozaev identities proved in \cite{RS1} and in \cite{B}, respectively. However, for $p \neq 2$, similar Pohozaev identities are not known.  In this paper, for the non-existence result, we use a Pucci-Serrin type identity given in \cite{RS}.

Now, we are in a position to state our first result.
\begin{theorem}[(ABC) type result]\label{ABC}
    Let $\Omega$ be a bounded $C^{1,\al}$ domain (for some $\al\in(0,1)$) in $\RN$ (with $N>p$), $1<q<p$, and $\eps\in(0,1]$. Then there exists $0<\La_\eps<\infty$ such that the following holds:
    \begin{enumerate}
         \item[\rm{(i)}] For $\la\in(0,\La_\eps)$, \eqref{main_PDE} admits a positive minimal weak solution. Moreover, minimal solutions are strictly increasing w.r.t. $\la$.
        \item[\rm{(ii)}] For $\la=\La_\eps$, \eqref{main_PDE} admits at least one positive weak solution.
        \item[\rm{(iii)}] For $\la>\La_\eps$, \eqref{main_PDE} does not admit any positive weak solution.
        \item[\rm{(iv)}] For $p>2$, $q\in[2,p)$ and $\la \le 0$, \eqref{main_PDE} does not admit any nontrivial solution in any star shaped domain.
    \end{enumerate}
\end{theorem}
In \cite[Theorem 1.1]{DFB}, da Silva, Fiscella and Viloria proved the existence of $\lambda_\ast>0$ such that for all $\lambda\in(0,\lambda_\ast)$, ($\tb{\mathcal{P}_{\la, 1}}$) has infinitely many nontrivial solutions $\{u_n\}$ with negative energy. In this paper, we note that further results can be obtained through the application of the Dual Fountain Theorem. Specifically in the appendix, we establish the existence of a sequence of nontrivial solutions of \eqref{main_PDE} with negative energy whose energy converges to zero (see Proposition \ref{inf_non_triv}).

Next, for $p \ge 2$, we investigate the multiplicity of positive solutions of \eqref{main_PDE} in the spirit of (ABC). In \cite{Tar}, Tarantello used Nehari manifold approach to show the multiplicity of positive solutions to the problem:
$$-\Delta u=|u|^{2^*-2}u+f\text{ in }\Om,\quad u=0\text{ on }\pa\Om.$$
The author partitioned the Nehari manifold $\La=\La^+\sqcup\La_0\sqcup\La^-$ into three disjoint sets, where the first solution lies in $\La^+$, and the second solution lies in $\La^-$.
Azorero and Peral in \cite{GaPa1} established the existence of a second positive solution for the $p$-Laplace operator when $p \in [2,3)$ and $q\in(1,p)$, as well as when $p\geq3$ and $q\in(\pst-\frac2{p-1},p)$. In the context of the fractional $p$-Laplace operator, Ye and Zhang \cite{YZ} demonstrated the existence of a second positive solution under the conditions $p\geq2$, $q \in (p-1, p)$, and $q>p^*_s-1$ where $p^*_s=\frac{Np}{N-ps}$ is the fractional critical exponent. Unlike the linear case $p=2$, for $p \neq 2$, these conditions on $q$ (depending on the range of $p$) arise due to the application of certain inequalities. In order to prove the existence of second positive solution, the authors in \cite{GaPa1, Tar, YZ} prove an energy estimate similar to \eqref{-1e1}.

Biagi and Vecchi in \cite{BV2} considered the following problem
\begin{equation}\label{BV2-p}
            -\Delta u+\eps(-\Delta)^s u=\lambda u^q+u^{2^*-1}\text{ in }\Om,\quad
            u>0\text{ in }\Om,\quad
            u=0\text{ in }\RR \setminus \Om,
\end{equation}
with $q \in (0,1)$. They showed that there exist $\la_\ast>0$, $\eps_0\in(0,1)$ such that \eqref{BV2-p} admits a second positive solution for any $\la\in(0,\la_\ast)$ and $\eps\in(0,\eps_0)$. For the existence of a second positive solution, they prove that there exist $\eps_0 \in (0,1)$ and $R_0\gg1$ such that the following energy estimate holds
\begin{equation}\label{-1e1}
    I_{\la, \eps}(u_{\la, \eps}+tRU_\eps)<I_{\la,\eps}(u_{\la, \eps})+\frac1NS_0^{\frac{N}{2}}, \text{ for every } R \ge R_0,\,  t \in [0,1], \, \eps \in (0, \eps_0),
\end{equation}
where $S_0$ is the classical Sobolev constant defined in \eqref{Sob_const}, $I_{\la,\eps}$ is the energy functional associated with \eqref{BV2-p}, $u_{\la, \eps}$ is a weak solution of \eqref{BV2-p}, and $U_\eps$ is the product of Aubin-Talenti bubble for the Laplace operator with a cutoff function supported inside $\Omega$. Then, the existence of a second positive solution is obtained by showing that its energy level differs from that of $u_{\la, \eps}$. For that, the authors used the Mountain pass theorem under a certain threshold $c \in \R$ where $I_{\la, \eps}$ satisfies the Palais-Smale (we call it by $\text{(PS)}_c$) condition. Unlike the case of purely local or non-nonlocal setup, the following difficulties appear for the mixed local-nonlocal operator in order to prove \eqref{-1e1}:

Observe that since $\eps_0<1$, their multiplicity result does not include the mixed local-nonlocal operator $-\Delta+(-\Delta)^s$. The primary difficulty arises because the mixed local-nonlocal operator is not scaling invariant. For $\eps =1$, \eqref{-1e1} may not hold  for every dimension $N$. For $p=2$ and $\eps=1$, one has the following estimate for $R>0$:
$$I_{\la}(u_{\la}+RU_\eps)\leq I_{\la}(u_{\la})+\frac1NS_0^{\frac{N}{2}}+C_1\eps^{2-2s}-C_2\eps^{\frac{N-2}{2}}+o(\eps^{\frac{N-2}{2}}),$$
which yields $$I_{\la}(u_{\la}+RU_\eps)<I_{\la}(u_{\la})+\frac1NS_0^{\frac{N}{2}},$$ provided both $N,s$ have some restricted ranges, i.e., $N<6$ and $s<\frac{6-N}{4}$.
In order to obtain \eqref{-1e1}, the existence of $\eps_0\in(0,1)$ plays a key role.

In this paper, for the existence of two positive solutions, we follow these steps:
\leftmargini=14mm
\begin{enumerate}
        \item[\rm{(Step 1)}] We prove that the first solution exists and it is a local minimizer of the energy functional associated with \eqref{main_PDE}.
        \item[\rm{(Step 2)}] We establish an energy estimate (see \eqref{energy estimate} and \eqref{energy estimate-2}), which is a generalization of \eqref{-1e1} to $p \geq 2$.
        \item[\rm{(Step 3)}] Next, we prove the local Palais-Smale condition of the energy functional associated with \eqref{main_PDE}.
        \item[\rm{(Step 4)}] Finally, we use the Mountain pass theorem to get a solution whose energy is different from the energy of the first solution.
    \end{enumerate}

We employ the following techniques to tackle the difficulties arising from scaling variant nonlinear operators:

\smallskip
\noi (a)  To compensate for the lack of scaling invariance of the mixed norm, we need to include $\eps$ with the fractional $p$-Laplace operator given in \eqref{main_PDE}. However, this inclusion introduces new challenges due to the $\eps$-dependency of any weak solution of \eqref{main_PDE}. For instance, every norm corresponding to a weak solution of \eqref{main_PDE} and the energy functional of \eqref{main_PDE} implicitly depends on $\eps$. To address this, we adopt the approach of Biagi and Vecchi \cite{BV1} by showing that every solution of \eqref{main_PDE} is uniformly bounded in $X_0$ (the solution space for \eqref{main_PDE}) and moreover in $L^{\infty}(\Omega)$, provided it lies in a ball (with radius independent of $\eps$) of $X_0$. Further, we find the first positive solution in such a ball. More specifically, we have shown the existence of $\eps\in(0,1]$ and $\la^{\#}>0$ (independent of $\eps$) such that for any $0<\la<\la^{\#}$, \eqref{main_PDE} has a positive solution $u_{\la,\eps}$ in the interior of $B_{r_0}$ (see Proposition \ref{B_r lemma}).

\smallskip

\noi (b) As discussed above, due to the nonlinear structure of the operator, for every $p \neq 2$, we did not get the whole range of $q\in(1,p)$ to obtain \eqref{energy estimate} and \eqref{energy estimate-2}. This limitation arises because certain inequalities valid in the linear case $p=2$ fail in the $p \neq 2$ case. In this paper, depending on the values of $p$, we provide a certain range of $q$ so that \eqref{energy estimate} and \eqref{energy estimate-2} hold for every $p \in [2, \infty)$ (see Proposition \ref{energy_prop}).

We say $(p,q)$ satisfy the condition \eqref{A_pq} if the following is satisfied:
\begin{equation}\label{A_pq}\tag{$\mathbf{A_{pq}}$}
    \begin{cases}
        2\leq p<3\text{ and }q\in(1,p);\\
        p\geq3\text{ and } q \in (\pst-\frac2{p-1},p).
    \end{cases}
\end{equation}

The following theorem states the multiplicity result.
\begin{theorem}\label{second_sol}
    Let  $p\in(1, N)$,  $q\in(1,p)$ and $\eps\in(0,1]$. Then there exists $\la^{\#}>0$, independent of $\eps$, such that for every $0<\la<\la^{\#}$, \eqref{main_PDE} has a positive weak solution $u_{\la,\eps}$.

   Moreover, if $p \in [2, \infty)$ and $(p,q)$ satisfies the condition \eqref{A_pq}, then for every $0<\la<\la^{\#}$, there exists $\eps_\la>0$ such that for any $\eps\in (0,\eps_\la)$,  \eqref{main_PDE} admits another positive weak solution $v_{\la, \eps}\neq u_{\la,\eps}$ .
\end{theorem}

\begin{remark}
    We remark that for $2\leq p<3$, $\eps_\la$ becomes independent of $\la$, and as a consequence, there exists $\eps_0>0$ such that \eqref{main_PDE} admits a second positive solution $v_{\la, \eps}\neq u_{\la,\eps}$ for any $\eps \in (0, \eps_0)$ (see Remark \ref{lambda-ind}).
\end{remark}
\begin{remark}
The range of $q$ for which we have obtained existence of two positive solutions in Theorem~\ref{second_sol} exactly matches with the range given in \cite[Theorem 1 and 2]{GaPa1}.
\end{remark}

The rest of the paper is organized as follows. In Section \ref{sec_prelms}, we list the notations, useful results and inequalities. In Section \ref{sec_exist_nonexist}, we prove Theorem \ref{ABC} and first part of Theorem \ref{second_sol}.  Section \ref{sec_multiplicity} is devoted to the proof of the second part of Theorem \ref{second_sol}. In the Appendix \ref{appndx}, we show the existence of a sequence of infinitely many nontrivial solutions with negative energy whose energy converges to zero.

\section{Preliminaries}\label{sec_prelms}
First, we fix the following notations and conventions to be used in this paper:

\noi\textbf{Notations:}
\begin{itemize}\label{notations}
    \item $P \lesssim_{a,b,c} Q$ represents that there exists a positive constant $C=C(a,b,c)$ such that $P \le C(a,b,c) Q$.
    \item We denote the best constant in the classical Sobolev inequality by $S_0$. i.e.,
    \begin{equation}\label{Sob_const}
    S_0=\inf_{u\in W_0^{1,p}(\Om),u\neq0}\frac{\|\nabla u\|_p^p}{\|u\|_{\pst}^p}.
    \end{equation}
    \item We have fixed $\la^\ast, \la^{\#},\la_{\ast\ast}$. $\la^\ast$ has been defined in \eqref{e5.01}.  $\la^{\#}$ and $\la_{\ast\ast}$ appear in Proposition \ref{B_r lemma} and $\la^{\#}$ is defined as $\min\{\la^\ast,\la_{\ast\ast}\}$.

    \item We consider the first Dirichlet eigenvalue, which is defined by $$\la_{1,\eps}:=\inf\{\rho_\eps(u)^p:u\in X_0, \|u\|_p^p=1\}.$$ Notice that $\la_{1,\eps_1}\leq\la_{1,\eps_2}$ for any $\eps_1\leq\eps_2$. $\la_{1,0}$ being the first Dirichlet eigenvalue of $\delp$, it is positive. Thus $\la_{1,\eps}>0$ for any $\eps>0$. Let $e_{1,\eps}$ be the corresponding eigenfunction. In \cite{GS}, Goel and Sreenadh proved that $\la_{1,\eps}$ is isolated and $e_{1,\eps}$ has constant sign, a.e.
    in $\Om$. So, we assume that it is positive, a.e. in $\Om$.
    \item We denote
    \begin{align*}
        &A_u(x,y):= |u(x)-u(y)|^{p-2}(u(x)-u(y)), \text{ and }\\
        &\mathcal{A}(u,v):= \iint\limits_{\RR \times \RR}\frac{A_u(x,y)(v(x)-v(y))}{|x-y|^{N+sp}} \dx \dy.
    \end{align*}
    \item By $u_{\la,\eps}$, we mean the solution obtained in Proposition \ref{B_r lemma}. We denote the positive minimal solution by $z_{\la,\eps}$. By $\hat{z}_{\la,\eps}$, we mean the solution of the truncated problem \eqref{trunc}. We denote the solution of the problem \eqref{sublinear} by $w_{\la,\eps}$. Finally, $v_{\la,\eps}$ is the second positive solution we get.
\end{itemize}
Let $\Omega \subset \RN$ be a bounded domain and $0 < \eps \le 1$. We consider the following function space
\begin{equation*}
    X_0:=\{u\in W^{1,p}(\RR): u|_{\Omega} \in W_0^{1,p}(\Omega),u\equiv0\text{ in }\RR\setminus\Omega\},
\end{equation*}
endowed with the norm
\begin{align*}
  \rho_{\eps}(u) := \left( \norm{\nabla \cdot}_p^p+ \eps[\cdot]_{s,p}^p \right)^{\frac{1}{p}},
\end{align*}
where $\norm{\nabla \cdot}_{p}$ is the $L^p$-norm of the gradient, and $[\cdot ]_{s,p}$ is the Gagliardo seminorm, defined as
\begin{align*}
    &\norm{\nabla u}_{p}^p=\int_{\Omega} |\nabla u|^p \dx \; \mbox{ and }\; [u]_{s,p}^p=  \iint\limits_{\RR \times \RR} \frac{|u(x)-u(y)|^p}{|x-y|^{N+sp}}\dy \dx.
\end{align*}
Observe that $\rho_{\eps}$ is equivalent to the usual norm $\rho$, defined as {\small $\rho(\cdot) := \left( \|\nabla \cdot\|_p^p+[\cdot]_{s,p}^p \right)^{\frac{1}{p}}$}. Further, using the embedding $W^{1,p}(\RR) \hookrightarrow W^{s,p}(\RR)$ and the  Poincar\'{e} inequality, $\norm{\Gr u}_p$ is also an equivalent norm in $X_0$.
The space $X_0$ is a ref{}lexive and separable Banach space with respect to $\rho_\eps(\cdot)$, $X_0$ is continuously embedded into $ L^{t}(\Omega)$ for $1\leq t\leq\pst$, and this embedding is compact when $t \neq \pst$.
\begin{definition}\label{def-WS}
A function $u\in X_0$ is called a weak solution of \eqref{main_PDE} if for every $v\in X_0$, it holds
\begin{equation}\label{weak}
\begin{aligned}
    \int_{\RR}|\nabla u|^{p-2}\nabla u\cdot\nabla v \dx+\eps\mathcal{A}(u,v)=\lambda\int_{\RR}|u|^{q-2}uv\dx+\int_{\RR}|u|^{\pst-2}uv\dx.
\end{aligned}
\end{equation}
\end{definition}

\begin{definition}
    A function $u\in X_0$ is called a weak supersolution of \eqref{main_PDE} if for every $v\in X_0$ with $v\geq0$ a.e. in $\Om$, it holds
\begin{equation}\label{e8}
    \int_{\RR}|\nabla u|^{p-2}\nabla u\cdot\nabla v \dx+\eps\AA(u,v)\geq\lambda\int_{\RR}|u|^{q-2}uv\dx+\int_{\RR}|u|^{\pst-2}uv\dx.
\end{equation}
Similarly, we say $u\in X_0$ is a subsolution of \eqref{main_PDE} if the reverse inequality holds in \eqref{e8}.
\end{definition}
Consider the following energy functional:
\begin{align}\label{I-la-eps}
    I_{\la,\eps}(u) := \frac1p\rho_\eps(u)^p-\frac{\lambda}{q}\|u_+\|_q^q-\frac{1}{\pst}\|u_+\|_{\pst}^{\pst}, \; \forall \, u \in X_0.
\end{align}
Observe that $I_{\la,\eps} \in C^1(X_0, \R)$, and every nonzero critical point of $I_{\la,\eps}$ is a non-negative nontrivial solution of \eqref{main_PDE}, and by the strong maximum principle \cite[Theorem~3.1]{BMV}(see the first paragraph of pg.10 in \cite{BMV}), it is a positive solution of \eqref{main_PDE}.
\begin{definition}
A sequence $\{u_n\}\subset X_0$ is said to be PS sequence of $I_{\la,\eps}$ at a level $c\in\R$ if
$$I_{\la,\eps}(u_n)\to c,\quad I_{\la,\eps}'(u_n)\to0\text{ in }X_0^* \quad\mbox{as}\quad n\to\infty.$$
We say that $I_{\la,\eps}$ satisfies Palais-Smale condition at the level $c$ (in short $(PS)_c$ condition) if $\{u_n\}$ is any PS sequence at a level $c$ then $\{u_n\}$ has convergent subsequence in $X_0$.
\end{definition}
 The following lemma states the classical comparison principle (see \cite[Proposition 4.1]{AC}).
\begin{lemma}\label{comparison-2}
    Let $u,v\in X_0$ satisfy $$\int_{\RR}|\nabla u|^{p-2}\nabla u\cdot\nabla \phi \dx+\eps\AA(u,\phi)\leq \int_{\RR}|\nabla v|^{p-2}\nabla v\cdot\nabla \phi \dx+\eps\AA(v,\phi)$$
    for every $\phi\in X_0, \phi \ge 0$. Then $u\leq v$ in $\Om$.
\end{lemma}
The following lemma lists some elementary inequalities to be used in the paper.

\begin{lemma}\label{q_ineq}
\noi {\rm{(i)}} For every $\eta,\xi\in\RR$,
\begin{align}
&(|\xi|^{t-2}\xi-|\eta|^{t-2}\eta)\cdot(\xi-\eta)\geq C(|\xi|+|\eta|)^{t-2}|\xi-\eta|^2, \text{ when } t>1. \label{elemt.ineq_1}
\end{align}
\noi {\rm{(ii)}} Suppose $1\leq t\leq3$. Then for every $a,b\geq0$ there exists $C>0$ such that
\begin{equation}\label{elemt.ineq_pst}
    \left|(a+b)^{t}-a^{t}-b^{t}-t ab(a^{t-2}+b^{t-2})\right|\leq\begin{cases}
        Cab^{t-1}&\text{if }a\geq b;\\
        Cba^{t-1}&\text{if }a\leq b.
    \end{cases}
\end{equation}
\noi {\rm{(iii)}} Suppose $t\geq3$. Then for every $a \ge 0$, $(1+a)^t\geq 1+a^t+ta+ta^{t-1}$.

\noi {\rm{(iv)}} Suppose $t\geq2$. Then for every $a \ge 0$, $(1+a)^t\geq 1+a^t+ta$.

\noi {\rm{(v)}} Let $2\leq t<3$. Then given $\zeta_1\in[t-1,2]$ there exists $C>0$ such that
\begin{align}\label{i-1}
   (1+a^2+2a\cos\theta)^{\frac t2}\leq1+a^t+ta\cos\theta+Ca^{\zeta_1},
\end{align}
for any $a\geq0$ and $\theta\in[0,2\pi]$.

\noi {\rm{(vi)}} Let $t\geq3$. Then there exists $C>0$ such that
\begin{align}\label{i-2}
   (1+a^2+2a\cos\theta)^{\frac t2}\leq1+a^t+ta\cos\theta+C(a^2+a^{p-1}),
\end{align}
for any $a\geq0$ and $\theta\in[0,2\pi]$.
\end{lemma}

\begin{proof}
  Proofs of (i),(iii) and (iv) are elementary. See \cite[Lemma 3.6]{BL} for (i). (ii) follows from \cite[Lemma 4]{BN2}. (v) and (vi) follow from \cite[Lemma A4]{GaPa1}.
\end{proof}
The next lemma helps us to interchange limit and integration.
\begin{lemma}\label{converg_limi}
Let $\{u_n\}$ be a PS sequence of $I_{\la,\eps}$, which is defined as in \eqref{I-la-eps}. Then there exists $u\in X_0$ such that, for every $\phi\in X_0$, the following convergences hold (up to a subsequence):
    \begin{align*}
&\mathrm{(i)}\,\int_{\Om}|\nabla u_n|^{p-2}\nabla u_n\cdot\nabla\phi\dx\to\int_{\Om}|\nabla u|^{p-2}\nabla u\cdot\nabla\phi\dx,\quad \mathrm{(ii)}\,\AA(u_n,\phi)\to\AA(u,\phi),\\
        &\mathrm{(iii)}\,\int_{\Om}|u_n|^{q-2}u_n\phi\dx\to\int_{\Om}|u|^{q-2}u\phi\dx,\\
        &\mathrm{(iv)}\,\int_{\Om}|u_n|^{\pst-2}u_n\phi\dx\to\int_{\Om}|u|^{\pst-2}u\phi\dx.
    \end{align*}
\end{lemma}
\begin{proof}
    Since $\{u_n\}$ is a PS sequence of $I_{\la,\eps}$, $\{u_n\}$ is bounded in $X_0$ and there exists $u\in X_0$ such that, up to a subsequence, $u_n\rightharpoonup u$ in $X_0$ and $u_n\to u$ a.e. in $\Om$ and $\nabla u_n\to \nabla u$ a.e. in $\Om$ (see \cite[Lemma 2.2]{DFB}). Now (i), (iii) and (iv) follow from \cite[Proposition A.8]{AP}. Using the arguments given in \cite[Lemma 2.4]{CP} or \cite[Eq. (2.17)]{DFB}, (ii) follows.
\end{proof}

In the following lemma, we state a particular version of the comparison principle from \cite[Theorem~2.1]{G}.
\begin{lemma}\label{comparison-1}
     Let $f$ be a non-negative continuous function and $\frac{f(t)}{t^{p-1}}$ is non-increasing for $t>0$. Assume that $u,v\in X_0$ are respectively positive subsolution and supersolution of the following problem:
    \begin{equation}\label{comp_eq}
        \begin{cases}
            \delp u+\eps\frap u=f(u)&\text{in }\Om,\\
            u=0&\text{in }\RR \setminus \Om.
        \end{cases}
    \end{equation}
    Then $u\leq v$ in $\Om$.
\end{lemma}
As a corollary to the above lemma, we can see that \eqref{comp_eq} has a unique positive solution. Next we prove method of subsolution and supersolution for \eqref{main_PDE}. Let $v\in X_0\cap L^\infty(\Om)$ be a positive function and $f_\la(t)=\la|t|^{q-2}t+|t|^{\pst-2}t$. We claim that the problem
\begin{equation}\label{sub_sup_pde}
    \delp u+\eps\frap u=f_\la(v)\text{ in }\Om,\quad u=0\text{ in }\rd \setminus \Om,
\end{equation} has a unique solution. Consider the functional $$K_{\la,\eps}(u)=\frac1p\rho_\eps(u)^p-\int_{\Om}f_\la(v)u_+\dx, \; \forall \, u \in X_0.$$ Since $f_\la$ is continuous and $v\in L^{\infty}(\Om)$, $K_{\la,\eps}$ is weakly lower semicontinuous and coercive. Thus, it has a global minimizer which is a non-negative solution of \eqref{sub_sup_pde}. Since, $f_\la(v)>0$, by strong maximum principle \cite[Theorem 3.1]{BMV}, $u>0$ a.e. in $\Om$. Now letting $f(t)=f_\la(v)$ a constant, by Lemma \ref{comparison-1}, $u$ is the unique positive solution. Moreover, by \cite[Theorem 4.1]{BMV}, $u\in L^{\infty}(\Om)$.
\begin{lemma}\label{sub_sup_lemma}
    Let $\underline{u}$ and $\ov{u}$ be positive subsolution and supersolution of \eqref{main_PDE} respectively. Suppose $\underline{u},\ov{u}\in L^{\infty}(\Om)$ and $\underline{u}\leq\ov{u}$ a.e. in $\Om$. Then there exists a positive weak solution $u\in X_0$ of \eqref{main_PDE} such that $\underline{u}\leq u\leq \ov{u}$ a.e. in $\Om$.
\end{lemma}
\begin{proof}
    Let $u_0=\underline{u}$. Given $u_n$, inductively we define $u_{n+1}\in X_0$ to be the unique positive solution of the problem
    \begin{equation}\label{baa1}
        \delp u_{n+1}+\eps\frap u_{n+1}=f_\la(u_n)\text{ in }\Om,\quad u_{n+1}=0\text{ in }\rd \setminus \Om.
    \end{equation}
    Applying Lemma~\ref{comparison-2} repeatedly, it follows that
    \begin{equation}\label{bab}
        \underline{u}=u_0\leq\cdots\leq u_n\leq u_{n+1}\leq\cdots\leq \ov{u}\text{ a.e. in }\Om.
    \end{equation}
    Notice from \eqref{baa1}, $$\rho_\eps(u_n)^p=\int_{\Om}f_\la(u_{n-1})u_n\dx\leq \int_{\Om}f_\la(\ov{u})\ov{u}\dx\leq C.$$
    Thus, $u_n$ is bounded in $X_0$. Hence there exists $u\in X_0$ such that, up to a subsequence, $u_n\rightharpoonup u$ in $X_0$ and $u_n\to u$ a.e. in $\Om$. By \eqref{bab}, $\underline{u}\leq u\leq \ov{u}\, $ a.e. in $\Om$.
    Finally to show that $u$ is a solution of \eqref{main_PDE}, we show that $\{u_n\}$ is a PS sequence of $I_{\la,\eps}$. By \eqref{baa1},
    \begin{align*}
        &I_{\la,\eps}(u_{n+1})=\frac1p\rho_{\eps}(u_{n+1})^p-\frac\la q\|u_{n+1}\|_q^q-\frac1\pst\|u_{n+1}\|_{\pst}^{\pst}\\
        &=\la\left(\frac1p\int_{\Om}u_n^{q-1}u_{n+1}\dx-\frac1q\|u_{n+1}\|_q^q\right)+\left(\frac1p\int_{\Om}u_n^{\pst-1}u_{n+1}\dx-\frac1{\pst}\|u_{n+1}\|_{\pst}^{\pst}\right).
    \end{align*}
    As $\underline{u}\leq u_n\leq u_{n+1}\leq\ov{u}$ a.e. in $\Om$, the sequence $\{I_{\la,\eps}(u_n)\}_n$ is a bounded sequence. Further, as $u_{n+1}$ solves \eqref{baa1}, for any $\phi\in X_0$,
    \begin{align*}
        I_{\la,\eps}'(u_{n+1})(\phi)&=\int_{\Om}|\nabla u_{n+1}|^{p-2}\nabla u_{n+1}\cdot\nabla\phi\dx+\eps\AA(u_{n+1},\phi)\\
        &\quad-\la \int_{\Om}u_{n+1}^{q-1}\phi\dx-\int_{\Om}u_{n+1}^{\pst-1}\phi\dx\\
        &=\la\left(\int_{\Om}u_{n}^{q-1}\phi\dx-\int_{\Om}u_{n+1}^{q-1}\phi\dx\right)\\
        &\quad+\left(\int_{\Om}u_{n}^{\pst-1}\phi\dx-\int_{\Om}u_{n+1}^{\pst-1}\phi\dx\right).
    \end{align*}
Applying  Vitali's convergence theorem in the RHS of the above equality we conclude that  $$I_{\la,\eps}'(u_{n+1})(\phi)\to 0\text{ as }n\to\infty.$$
    Thus, applying Bolzano-Weierstrass theorem on $\{I_{\la,\eps}(u_n)\}_n$, $\{u_n\}$ has a subsequence, still denoted by $u_n$, which is a PS sequence of $I_{\la,\eps}$. Hence by Lemma \ref{converg_limi}, $u$ weakly solves  \eqref{main_PDE}.
\end{proof}
\section{Existence and nonexistence of solutions}\label{sec_exist_nonexist}

We begin by discussing the nonexistence of nontrivial solutions of \eqref{main_PDE} for $\la\le0$, using the following result from \cite[Proposition~1.4]{RS}.

\begin{proposition}\label{RS_prop}
Let $E$ be a Banach space contained in $L^1_{\text{loc}}(\RR)$, and $\|\cdot\|$ be a seminorm in $E$. Assume that for some $\alpha>0$ the seminorm $\|\cdot\|$ satisfies \begin{equation}\label{e3}
       w_\tau\in E\text{ and } \|w_\tau\|\leq\tau^{-\alpha}\|w\|, \; \forall \,w\in E, \tau>1,
\end{equation}
where $w_\tau(x)=w(\tau x)$. Let $\Omega\subset\RR$ be a bounded star-shaped domain with respect to the origin. For $p>1$ and $f\in C^{0,1}_{\text{loc}}(\ov{\Omega}\times\R)$, consider the energy functional $$\mathcal{E}(u)=\frac1p\|u\|^p-\int_{\Omega}F(x,u)\dx,$$ where $F(x,u)=\int_0^uf(x,t)\,dt$, and  $f$ satisfies the following inequality
   \begin{equation}\label{e4}
       \alpha tf(x,t)>NF(x,t)+x\cdot F_x(x,t)\quad\forall x\in\Omega, t\neq0.
   \end{equation}
Let $u \in E$ be a critical point of $\mathcal{E}$ satisfying $u=0$ in $\RR\setminus\Omega$. If $u\in L^{\infty}(\Omega)\cap W^{1,r}(\Omega)$ for some $r>1$, then $u\equiv0$.
\end{proposition}

\begin{proposition}\label{nonexistence}
Let $p>2$, $q\in[2,p)$ and $\eps\in(0,1]$. Let $\Omega$ be a bounded and star-shaped domain with $\partial\Omega$ being $C^{1,\alpha}$, for some $\alpha\in(0,1)$. Then, for $\lambda\leq 0$, \eqref{main_PDE} does not admit any nontrivial solution.
\end{proposition}
\begin{proof}
 Let $u\in X_0$ be a solution of \eqref{main_PDE}. By \cite{AC}, $u\in C^1(\ov{\Om})$. We see that, for $\tau>1$,
\begin{align*}
    \rho_\eps(u_\tau)^p=\|\nabla u_\tau\|_p^p+\eps[u_\tau]_{s,p}^p= \tau^{-(N-p)}\|\nabla u\|_p^p+\eps\tau^{-(N-sp)}[u]_{s,p}^p \leq \tau^{-(N-p)}\rho_\eps(u)^p.
\end{align*}Thus $\rho_\eps$ satisfies \eqref{e3} with $\alpha=\frac{N-p}{p}$. It is also easy to check that $f_\la(t)=\lambda|t|^{q-2}t+|t|^{\pst-2}t\in C^{0,1}_{\text{loc}}(\R)$.

\vspace{4pt}
\noi\textbf{Case-1:} Suppose $\la<0$. Since $\pst>q$,
\begin{align*}
    \alpha tf_\la(t)=N\left(\lambda\frac{|t|^q}{\pst}+\frac{|t|^{\pst}}{\pst}\right)>N\left(\lambda\frac{|t|^q}{q}+\frac{|t|^{\pst}}{\pst}\right)=N F_\la(t)+x\cdot (F_\la)_x(t),
\end{align*}
where $F_\la(t)=\int_0^tf_\la(s)\ds$. Thus $f_\la$ satisfies \eqref{e4}. Hence, by Proposition \ref{RS_prop}, the only bounded solution is the trivial solution. Therefore, $u \equiv 0$, as required.

\vspace{4pt}
\noi\textbf{Case-2:} Suppose $\la=0$. Then we have $$\alpha tf_0(t)=\frac{N}{\pst}|t|^{\pst}=N F_0(t)+x\cdot (F_0)_x(t).$$
So, it does not satisfy the inequality \eqref{e4}. Define $\Phi(\cdot):=\frac1p\rho_{\eps}(\cdot)^p$ and
$$I_\tau:=\tau^\alpha\langle D\phi(u),u_\tau\rangle=\tau^{\alpha}\left(\int_{\Om}|\nabla u|^{p-2}\nabla u\cdot\nabla u_\tau\dx+\eps\AA(u,u_\tau)\right).$$
Using (4.3) from the proof of \cite[Proposition 1.4]{RS}, we have
$$\alpha\int_{\Om}uf_0(u)\dx=\int_{\Om}(N F_0(u)+x\cdot (F_0)_x(u))\dx+\frac{d}{d\tau}\bigg|_{\tau=1^+}I_{\tau}.$$
Which implies that
\begin{equation}\label{I_tau=0}
\frac{d}{d\tau}\bigg|_{\tau=1^+}I_{\tau}=0.
\end{equation}
We also have, by \cite[Lemma 4.1]{RS}
\begin{align*}
I_\tau&\leq p\tau^\alpha \Phi(u)^{\frac1{p'}}\Phi(u_\tau)^{\frac1p}=\tau^\alpha\rho_\eps(u)^{\frac p{p'}}\rho_\eps(u_\tau)=\tau^\alpha I_1^{\frac1{p'}}\rho_{\eps}(u_\tau)\\&\leq \frac 1{p'}I_1+\frac1p\tau^{N- p}\rho_\eps(u_\tau)^p=\frac 1{p'}I_1+\frac1p \left(\|\nabla u\|_p^p+\eps\tau^{sp-p}[u]_{s,p}^p\right).
\end{align*}
Thus,
\begin{align*}
    I_1-I_\tau\geq \frac{\eps}p[u]_{s,p}^p(1-\tau^{sp-p}).
\end{align*}
Notice that $\lim_{h\to0}\frac{1-(h+1)^{sp-p}}{h}=p-ps.$ Therefore, 
\begin{align*}
    -\frac{d}{d\tau}\bigg|_{\tau=1^+}I_{\tau}&\geq (1-s)\eps[u]_{s,p}^{p},
\end{align*}
i.e., $\frac{d}{d\tau}\bigg|_{\tau=1^+}I_{\tau}\leq -(1-s)\eps[u]_{s,p}^p$.  Thus, by \eqref{I_tau=0}, $[u]_{s,p}=0$. Hence, $u\equiv0$.
\end{proof}

\begin{remark}
The above proof also indicates that for every $q<\pst$ and $\la\leq0$, \eqref{main_PDE} does not have any nontrivial solution on a bounded, star-shaped domain.
\end{remark}
Next, we proceed to prove the existence of a positive solution to \eqref{main_PDE} provided $\la$ lies in a certain range. Define the closed ball $B_r:=\{u\in X_0: \rho_\eps(u)\leq r\}$ in $X_0$.
\begin{proposition}\label{B_r lemma}
    Let $q\in(1,p)$ and $\eps\in(0,1]$. Then there exist $r_0>0$ independent of $\eps$ such that for every $0<r\leq r_0$, there exists $\la^{\#}>0$, independent of $\eps$, such that for every $0<\la<\la^{\#}$, \eqref{main_PDE} has a positive solution $u_{\la,\eps}$ in the interior of $B_{r}$.
\end{proposition}
\begin{proof}
By the Sobolev embedding $X_0 \hookrightarrow L^{p^*}(\Omega)$, we can choose $r_0>0$ and $\delta_0>0$  small enough such that, for every $0<r\leq r_0$,
    \begin{equation}\label{lbound-1}
    \begin{aligned}
        &\frac1p\rho_\eps(u)^p-\frac1{\pst}\|u_+\|_{\pst}^{\pst}\geq\frac1p\rho_\eps(u)^p-C_1\rho_\eps(u)^{\pst}\geq \left\{\begin{array}{ll}
             0, & \forall\,u\in B_{r}; \\
             2\delta_0, & \forall\,u\in \pa B_{r}.  \\
             \end{array} \right.
    \end{aligned}
    \end{equation}
Let $I_{\la,\eps}$ be as defined in \eqref{I-la-eps}. Therefore, $I_{\la,\eps}(u)\geq -\frac{\la}{q}\|u_+\|_q^q\geq-\frac{\la C}{q}r_0^q>-\infty$ for every $u\in B_{r}$. Thus $I_{\la,\eps}$ is bounded from below on $B_{r}$. Define, $c_{\la,\eps}:=\inf_{B_{r_0}}I_{\la,\eps}(u).$ For any $v\neq0$ in $X_0$ and $t>0$,
    \begin{align*}
        I_{\la,\eps}(tv)&<\frac{t^p}p\rho_1(v)^p-\frac{\la t^q}{q}\|v_+\|_q^q.
    \end{align*}
\noi It is easy to see that for small $t$, $I_{\la,\eps}(tv)<0$. Hence $-\infty<c_{\la,\eps}<0$. Consider a minimizing sequence $u_{n,\eps}\in B_{r}$ i.e. $I_{\la,\eps}(u_{n,\eps})\to c_{\la,\eps} \text{ as } n\to\infty.$
We claim that $\rho_\eps(u_{n,\eps})\leq {r}-\eps_0$ for every $n$ and for some $\eps_0$ independent of $n$. If not, $\rho_\eps(u_{n,\eps})\to {r}$. Thus, by Sobolev inequality,
\begin{align*}
 c_{\la,\eps}=\lim_{n\to\infty}I_{\la,\eps}(u_{n,\eps})&\geq \lim_{n\to\infty}\left(\frac1p\rho_\eps(u_{n,\eps})^p-C_1\rho_\eps(u_{n,\eps})^{\pst}-\la C_2\rho_\eps(u_{n,\eps})^{q}\right)\\
&\geq 2\delta_0-\la C_2 r^q.
\end{align*}
Choosing $\la_{\ast\ast}>0$ such that for every $0<\la<\la_{\ast\ast}$ it holds $2\delta_0-\la C_2 r^q>\delta_0$. Thus we get  $$0>c_{\la,\eps}>\delta_0>0,$$
which yields a contradiction and hence $u_{n,\eps}\in B_{r-\eps_0}$ for every $n$. Let $0<\eps_1<\eps_0$, then $B_{r-\eps_0}\subset B_{r-\eps_1}$. Applying Ekeland's variational principle on the complete metric space $B_{r-\eps_1}$ (w.r.t Euclidean metric), we get
  $$c_{\la,\eps}\leq I_{\la,\eps}(u_{n,\eps})\leq c_{\la,\eps}+\frac{1}{n},$$  $$ I_{\la,\eps}(u_{n,\eps})\leq I_{\la,\eps}(v)+\frac1n\rho_{\eps}(u_{n,\eps}-v), \; \forall \, v\in B_{r-\eps_1}, \, v\neq u_{n,\eps}.$$
  Further, $$I_{\la,\eps}(v)=I_{\la,\eps}(u_{n,\eps})+I_{\la,\eps}'(u_{n,\eps})(v-u_{n,\eps})+o(\rho_\eps(v-u_{n,\eps})).$$
  Let $w\in X_0$ and $t>0$ be such that $\rho_\eps(w)=1$ and $v=u_{n,\eps}+tw\in B_{r-\eps_1}$. Therefore, using the above relation we get $$-\frac tn\leq I_{\la,\eps}(u_{n,\eps}+tw)-I_{\la,\eps}(u_{n,\eps})=tI_{\la,\eps}'(u_{n,\eps})w+o(t).$$
  Dividing by $t$ and letting $t\to0$, we obtain $$-\frac1n\leq I_{\la,\eps}'(u_{n,\eps})w.$$Replacing $-w$ by $w$, we conclude $$\|I_{\la,\eps}'(u_{n,\eps})\|\leq\frac1n.$$
Thus, $u_{n,\eps}$ is a PS sequence of $I_{\la,\eps}$ at the level $c_{\la,\eps}<0$. By \cite[Lemma 2.4(ii)]{DFB}, we know that  $I_{\la,\eps}$ satisfies (PS)$_c$ for any
\begin{equation}\label{PSc}
    c<\frac1N S_0^{\frac{N}{p}}-|\Omega|\left(\frac1p-\frac1{\pst}\right)^{-\frac{q}{\pst-q}}\left(\lambda\left(\frac1q-\frac1p\right)\right)^{\frac{\pst}{\pst-q}}.
\end{equation}
Choose
\begin{equation}\label{e5.01}
    \lambda^\ast:=\frac{\left(\frac{1}{N |\Omega|}   S_0^{\frac{N}{p}}\right)^{\frac{\pst-q}{\pst}}\left(\frac1p-\frac1{\pst}\right)^{\frac{q}{\pst}}}{\frac1q-\frac1p},
\end{equation}
so that for every $\lambda\in(0,\lambda^\ast),$
$$ \frac1N S_0^{\frac{N}{p}}>|\Omega|\left(\frac1p-\frac1{\pst}\right)^{-\frac{q}{\pst-q}}\left(\lambda\left(\frac1q-\frac1p\right)\right)^{\frac{\pst}{\pst-q}}.$$
Hence, for any $0<\la<\la^{\#}:=\min\{\la^\ast,\la_{\ast\ast}\},$ there exists $u_{\la,\eps}\in B_{r-\eps_0}$ such that $u_{n,\eps}\to u_{\la,\eps}$ in $X_0$, $I_{\la,\eps}(u_{\la,\eps})=c_{\la,\eps}$ and $u_{\la,\eps}$ is a critical point of $I_{\la,\eps}$. Hence $u_{\la,\eps}$ is a positive solution of \eqref{main_PDE}.
\end{proof}

\vspace{2mm}

Next, we consider the following purely sublinear problem:
\begin{equation}\label{sublinear}
    \delp u+\eps\frap u=\lambda|u|^{q-2}u\text{ in }\Omega,\quad u>0 \text{ in }\Omega,\quad
    u=0\text{ in }\RR\setminus\Omega.
\end{equation}
Notice that $$a_0:=\lim_{t\downarrow0}\frac{\la t^{q-1}}{t^{p-1}}=\infty, \text{ and } a_{\infty}:=\lim_{t\to\infty}\frac{\la t^{q-1}}{t^{p-1}}=0.$$
Thus, by \cite[(1.7)]{BMV}, $\la_1(\delp+\eps\frap-a_0)=-\infty$ and recalling the notations from Section \ref{sec_prelms}, $$\la_1(\delp+\eps\frap-a_\infty)=\la_1(\delp+\eps\frap)=\la_{1,\eps}>\la_{1,0}>0.$$
With $f(t)=\la|t|^{q-2}t$, we can apply \cite[Theorem~1.2]{BMV} (since $\eps$ is immaterial) to get the unique positive solution $w_{\la,\eps}$ to \eqref{sublinear}. Define the energy functional corresponding to \eqref{sublinear} by
\begin{equation}\label{J-la-eps}
    J_{\la,\eps}(u)=\frac1p\rho_\eps(u)^p-\frac{\la}{q}\|u\|_{q}^q, \; \forall \, u \in X_0.
\end{equation}

By \cite[Proposition~6.2]{BMV}, $w_{\la,\eps}$ is in fact the unique global minimizer of $J_{\la,\eps}$ and $J_{\la,\eps}(w_{\la,\eps})<0$. Applying \cite[Theorem~4.1]{BMV} and \cite[Theorem~1.1 and Theorem~1.2]{AC}, we also have $w_{\la,\eps}\in C^{1,\al}(\overline{\Om})$ for some $\alpha\in(0,1)$ and $\pa_\nu w_{\la,\eps}<0$ on $\pa\Om$.
\begin{lemma}\label{Laepsinfty}
For $\eps \in (0, 1]$, we define $$\La_\eps=\sup\{\la:\eqref{main_PDE}\text{ has a positive solution}\}.$$ Then $0<\La_\eps<\infty$.
\end{lemma}
\begin{proof}
    By Proposition \ref{B_r lemma}, $\La_\eps\geq\la^{\#}>0$. We prove $\La_\eps<\infty$ using the method of contradiction. Since the first Dirichlet eigenvalue $\la_{1,\eps}$ is isolated, there exists $\Tilde{\la}>\la_{1,\eps}$ such that
\begin{equation}\label{Eigenvalue_prob}
    \begin{cases}
        \delp u+\eps\frap u=\Tilde{\la} |u|^{p-2}u&\text{ in }\Om,\\
        u=0&\text{ in }\RR \setminus \Om,
    \end{cases}
\end{equation}
has no solution.
Suppose $\La_\eps=\infty$. Then there exists $\ov{\la}$ such that $$\ov{\la}t^{q-1}+t^{\pst-1}>\Tilde{\la}t^{p-1},\quad\forall \, t>0,$$
and $(\tb{\PP_{\ov{\la},\eps}})$ has a positive solution $u_{\ov{\la},\eps}$. 
Now notice that $$(\delp+\eps\frap)(\delta e_{1,\eps})=\delta^{p-1}\la_{1,\eps}e_{1,\eps}^{p-1}<\Tilde{\la} (\delta e_{1,\eps})^{p-1}.$$
By the definition of $\ov{\la}$,
\begin{align*}
    (\delp+\eps\frap)u_{\ov{\la},\eps}=\ov{\la}u_{\ov{\la},\eps}^{q-1}+u_{\ov{\la},\eps}^{\pst-1}>\Tilde{\la}u_{\ov{\la},\eps}^{p-1}.
\end{align*}
The above is true for all $\delta>0$. Thus by Lemma~\ref{comparison-1}, $\delta e_{1,\eps}\leq u_{\bar\la,\eps}$. Now we choose $\delta>0$ small enough such that $\delta e_{1,\eps}< u_{\bar\la,\eps}$.

Thus, by the method of subsolution-supersolution, there exists a nontrivial solution to \eqref{Eigenvalue_prob}, which is a contradiction. Hence, $\La_\eps$ has to be finite for every $\eps \in (0,1]$.
\end{proof}
For the positive minimal solution of \eqref{main_PDE} we require the following lemma.
\begin{lemma}\label{mis_lem}
    Let $f_\la(t)=\la t^{q-1}+t^{\pst-1}$ for $\la>0,t>0$. Then for any $0<\la<\la'<\infty$ and $M>0$, there exists $\beta_0>1$ such that $f_\la(\beta_0 t)\leq f_{\la'}(t)$ for $0<t\leq M$.
\end{lemma}
\begin{proof}
    The proof follows from  \cite[Lemma 3.3]{YZ}.
\end{proof}

\begin{lemma}\label{minimal_lemma}
    Let $q\in(1,p)$ and $\eps\in(0,1]$. For every $0<\la<\La_\eps$, \eqref{main_PDE} has a minimal positive weak solution $z_{\la,\eps}$. Further, the minimal solutions are strictly increasing with respect to $\la$, i.e., if $0<\la<\la'<\La_\eps$,  then $z_{\la,\eps}< z_{\la',\eps}$ a.e. in $\Omega$.
\end{lemma}
\begin{proof}
    Let $0<\la<\La_\eps$. By the definition of $\La_\eps$, there exists $\la'\in(\la,\La_\eps)$ such that $(\PP_{\la',\eps})$ has a positive solution. Let $\ov{z}_{\la',\eps}$ be any positive solution of $(\PP_{\la',\eps})$. Since \eqref{sublinear} has a positive solution $w_{\la,\eps}$ for any $\la>0$, using Lemma \ref{comparison-1}, it follows that $w_{\la,\eps}\leq w_{\la',\eps}\leq \ov{z}_{\la',\eps}$ a.e. in $\Omega$.
    We define $u_0:=w_{\la,\eps}$ and given $u_n$, we define $u_{n+1}$ be the unique solution of \eqref{baa1}. As $w_{\la,\eps}$ and $\ov{z}_{\la',\eps}$ are respectively subsolution and supersolution of \eqref{main_PDE}, by Lemma \ref{sub_sup_lemma}, we have $w_{\la,\eps}\leq u_n\leq u_{n+1}\leq \ov{z}_{\la',\eps}$ for all $n$ and $u_n$ converges to a positive weak solution $z_{\la,\eps}$ of \eqref{main_PDE} satisfying $w_{\la,\eps}\leq z_{\la,\eps}\leq \ov{z}_{\la',\eps}$. Furthermore, for any solution $u$ of \eqref{main_PDE}, by Lemma \ref{comparison-2}, we have $u_n\leq u$ for all $n$ which in turn implies $z_{\la,\eps}\leq u$. Since $u$ is arbitrary, $z_{\la,\eps}$ is in fact the minimal positive weak solution of \eqref{main_PDE}. Thus, for any $0<\la<\La_\eps$, \eqref{main_PDE} has a minimal positive solution $z_{\la,\eps}$. Now, replacing $\ov{z}_{\la',\eps}$ by the minimal positive solution $z_{\la',\eps}$ of $(\PP_{\la',\eps})$, we get that $z_{\la,\eps}\leq z_{\la',\eps}$. Hence, minimal solutions are increasing with respect to $\la$.
To get the strict inequality, we apply Lemma \ref{mis_lem} and comparison principle. Using \cite[Theorem~4.1 and  Remark 4.2]{BMV}, $z_{\la',\eps}\in L^{\infty}(\Om)$. Set $M=\|z_{\la',\eps}\|_{\infty}$, and $f_\la(t)=\la t^{q-1}+t^{\pst-1}$ for $\la>0,t>0$. By Lemma \ref{mis_lem}, there exists $\beta_0>1$ such that the following holds weakly
    \begin{align}\label{eqn-1}
        (\delp+\eps\frap)z_{\la',\eps}=f_{\la'}(z_{\la',\eps})\geq f_{\la'}(z_{\la,\eps})\geq f_{\la}(\beta_0 z_{\la,\eps}).
    \end{align}
    Moreover, the following equation holds weakly
    \begin{align}\label{eqn-2}
        (\delp+\eps\frap)(\beta_0^{\frac{q-1}{p-1}} z_{\la,\eps})=\beta_0^{q-1}f_{\la}(z_{\la,\eps}).
    \end{align}
    In view of \eqref{eqn-1} and \eqref{eqn-2}, and the fact that $\beta_0>1$, the following holds weakly
    \begin{align*}
        (\delp+\eps\frap)z_{\la',\eps}\geq (\delp+\eps\frap)(\beta_0^{\frac{q-1}{p-1}} z_{\la,\eps}).
    \end{align*}
    By the comparison principle (Lemma \ref{comparison-2}), we get $z_{\la',\eps}\geq \beta_0^{\frac{q-1}{p-1}} z_{\la,\eps}>z_{\la,\eps}$ a.e. in $\Om$.
\end{proof}

\begin{lemma}\label{Laeps}
    Let $q\in(1,p)$ and $\eps\in(0,1]$. Then $(\PP_{\La_{\eps},\eps})$ has a positive solution.
\end{lemma}
\begin{proof}
\

\noindent
   {\bf Claim.} \emph{For any $0<\la<\La_\eps$, \eqref{main_PDE} has a positive solution with negative energy}.

   Indeed, to prove the claim in view of Proposition \ref{B_r lemma}, we only need to investigate the case when $\la \in [\la^{\#}, \La_\eps)$. Let $\la^{\#}\le \la<\la'<\La_\eps$. Since $w_{\la,\eps}$ and $\ov{z}_{\la',\eps}$ are respectively subsolution and supersolution of \eqref{main_PDE} (see Lemma \ref{minimal_lemma}), we consider the following truncated function associated with $f_{\la}(t)=\la t^{q-1}+t^{\pst-1}$:
\begin{align*}
  \hat{f}_\la(x,t)=\begin{cases}
    \la \ov{z}_{\la',\eps}^{q-1}+\ov{z}_{\la',\eps}^{\pst-1},&t\geq \ov{z}_{\la',\eps}(x);\\
    \la t^{q-1}+t^{\pst-1},&w_{\la,\eps}(x)\leq t\leq \ov{z}_{\la',\eps}(x);\\
    \la w_{\la,\eps}^{q-1}+w_{\la,\eps}^{\pst-1},&t\leq w_{\la,\eps}(x).
\end{cases}
\end{align*}
Next, we consider the following problem
\begin{equation}\label{trunc}
    \delp u+\eps\frap u=\hat{f}_{\la}(x,u)\text{ in }\Om,\quad u\geq0\text{ in }\Om,\quad u=0\text{ in }\RR \setminus \Om.
\end{equation}
with the energy functional
$$\hat{I}_{\la,\eps}(u):=\frac1p\rho_{\eps}(u)^p-\int_{\Om}\hat{F}_\la(x,u(x))\dx,\quad \text{where }\hat{F}_\la(x,u(x))=\int_0^{u}\hat{f}_\la(x,t)\dt.$$
By \cite{AC}, $w_{\la,\eps},\, \ov{z}_{\la',\eps}\in C^1(\ov{\Om})$.
We first claim that $\hat{I}_{\la,\eps}$ is weakly lower semicontinuous and coercive on $X_0$. To check it, we suppose $u_n\rightharpoonup u$ in $X_0$. Thus, $u_n\to u$ in $L^1(\Om)$. Using $\hat{f}_\la\in L^{\infty}(\Om\times\R)$, we obtain
\begin{align*}
\int_{\Om}|\hat{F}_\la(x,u_n(x))-\hat{F}_\la(x,u(x))|\dx&\leq \int_{\Om}\left|\int_{u(x)}^{u_n(x)}\hat{f}_\la(x,t)\dt\right|\dx\\&\leq C\int_{\Om}|u_n-u|\dx=o_n(1).
\end{align*}
Since $\norm{\Gr \cdot}_p$ and $[\cdot]_{s,p}$ are weakly lower semicontinuous, we conclude that $\hat{I}_{\la,\eps}$ is weakly lower semicontinuous on $X_0$. Using $\hat{f}_{\la}\in L^{\infty}$ again, $|\hat{F}_{\la}(x,t)|\leq C|t|$ for some $C>0$ and all $t\in \R$. For coercivity, we observe that by Young's inequality and the Sobolev inequality, for any $u\in X_0$, for some $C_1,C_2>0$, $$\hat{I}_{\la,\eps}(u)\geq C_1\rho_\eps(u)^p-C_2.$$ Thus, $\hat{I}_{\la,\eps}$ achieves its global minimum at some point $\hat{z}_{\la,\eps}$. Thus $\hat{z}_{\la,\eps}$ solves \eqref{trunc} and again by \cite{AC}, $\hat{z}_{\la,\eps}\in C^1(\ov{\Om})$.

{\bf Subclaim}: We claim that  $w_{\la,\eps}\leq\hat{z}_{\la,\eps}\leq \ov{z}_{\la',\eps}$.

To see this, we proceed as follows: since $w_{\la,\eps}$ is a subsolution and $\hat{z}_{\la,\eps}$ is a solution of \eqref{trunc}, we have
\begin{align*}
    &\int_{\Om}|\nabla w_{\la,\eps}|^{p-2}\nabla w_{\la,\eps}\cdot\nabla(w_{\la,\eps}-\hat{z}_{\la,\eps})^+\dx+\eps\AA(w_{\la,\eps},(w_{\la,\eps}-\hat{z}_{\la,\eps})^+)\\
    &\leq \int_{\Om}f_\la(w_{\la,\eps})(w_{\la,\eps}-\hat{z}_{\la,\eps})^+\dx,\\
    &\int_{\Om}|\nabla \hat{z}_{\la,\eps}|^{p-2}\nabla \hat{z}_{\la,\eps}\cdot\nabla(w_{\la,\eps}-\hat{z}_{\la,\eps})^+\dx+\eps\AA(\hat{z}_{\la,\eps},(w_{\la,\eps}-\hat{z}_{\la,\eps})^+)\\
    &= \int_{\Om}f_\la(w_{\la,\eps})(w_{\la,\eps}-\hat{z}_{\la,\eps})^+\dx.
\end{align*}
In the right hand side of the second line, we used the fact that whenever $w_{\la,\eps}\geq \hat{z}_{\la,\eps}$, $\hat{f}_\la(x,\hat{z}_{\la,\eps})=f_\la(w_{\la,\eps})$. Subtracting both expressions,
\begin{equation}\label{caa1}
    \begin{split}
    \mathcal{B}+\eps \mathcal{D}:=\int_{\Om}(|\nabla w_{\la,\eps}|^{p-2}\nabla w_{\la,\eps}-|\nabla \hat{z}_{\la,\eps}|^{p-2}\nabla \hat{z}_{\la,\eps})\cdot\nabla(w_{\la,\eps}-\hat{z}_{\la,\eps})^+\dx\\+\eps\left(\AA(w_{\la,\eps},(w_{\la,\eps}-\hat{z}_{\la,\eps})^+)-\AA(\hat{z}_{\la,\eps},(w_{\la,\eps}-\hat{z}_{\la,\eps})^+)\right)\leq0.
\end{split}
\end{equation}
In view of the inequality \eqref{elemt.ineq_1}, we get
\begin{equation}\label{caa2}
    \mathcal{B}\gtrsim \int_{\Om}(|\nabla w_{\la,\eps}|+|\nabla\hat{z}_{\la,\eps}|)^{p-2}|\nabla(w_{\la,\eps}-\hat{z}_{\la,\eps})^+|^2\dx.
\end{equation}
To estimate $\mathcal{D}$, we set $\xi=w_{\la,\eps}(x)-w_{\la,\eps}(y)$, $\eta=\hat{z}_{\la,\eps}(x)-\hat{z}_{\la,\eps}(y)$ and $w=(w_{\la,\eps}-\hat{z}_{\la,\eps})^+$. Further, we notice that
\begin{equation}\label{daa}
    (f(x)-f(y))(f^+(x)-f^+(y))\geq |f^+(x)-f^+(y)|^2.
\end{equation}
Observe that, by \eqref{daa}, $f(x)-f(y)$ and $f^+(x)-f^+(y)$ always have same sign. Now, for $p <2$ we estimate
\begin{align*}
    &[(w_{\la,\eps}-\hat{z}_{\la,\eps})^+]_{s,p}^p= \iint\limits_{\R^{2N}}|(w(x)-w(y))^2|^{\frac p2}\frac{\dx\dy}{|x-y|^{N+sp}}\\
    &\leq\iint\limits_{\R^{2N}}|(\xi-\eta)(w(x)-w(y))|^{\frac p2}\frac{\dx\dy}{|x-y|^{N+sp}}\\
    &\lesssim \iint\limits_{\R^{2N}}|(|\xi|^{p-2}\xi-|\eta|^{p-2}\eta)(w(x)-w(y))(|\xi|+|\eta|)^{2-p}|^{\frac p2}\frac{\dx\dy}{|x-y|^{N+sp}}\\
    &\leq \left(\,\iint\limits_{\R^{2N}}\frac{|(|\xi|^{p-2}\xi-|\eta|^{p-2}\eta)(w(x)-w(y))|}{|x-y|^{N+sp}}\dx\dy\right)^{\frac p2} \\
    &\quad\left(\,\iint\limits_{\R^{2N}}\frac{(|\xi|+|\eta|)^{p}}{|x-y|^{N+sp}}{\dx\dy}\right)^{\frac {2-p}2}\\
    &\lesssim \mathcal{D}^{\frac p2} \left([w_{\la,\eps}]_{s,p}^p+[\hat{z}_{\la,\eps}]_{s,p}^p\right)^{\frac {2-p}2},
\end{align*}
where the first inequality follows from \eqref{daa}, the second inequality follows from \eqref{elemt.ineq_1} and applying the H\"{o}lder's inequality with exponents $(\frac 2p,\frac 2{2-p})$ yields the penultimate inequality. In the first term of the penultimate inequality, we see that $|\xi|^{p-2}\xi-|\eta|^{p-2}\eta$ and $w(x)-w(y)$ have the same sign, indeed as discussed above, $w(x)-w(y)$ has the same sign as $(w_{\la,\eps}-\hat{z}_{\la,\eps})(x)-(w_{\la,\eps}-\hat{z}_{\la,\eps})(y)$, which in turn yields that $|\xi|^{p-2}\xi-|\eta|^{p-2}\eta$ has the same sign as $w(x)-w(y)$. This allows us to remove modulus in the first term to get $\mathcal{D}$. Hence, for $p<2$, we have
\begin{equation}\label{caa3}
    \mathcal{D}\gtrsim \frac{[(w_{\la,\eps}-\hat{z}_{\la,\eps})^+]_{s,p}^2}{\left([w_{\la,\eps}]_{s,p}^p+[\hat{z}_{\la,\eps}]_{s,p}^p\right)^{\frac {2-p}p}}.
\end{equation}
For $p \ge 2$, using same arguments as in \cite[Proposition 5.12]{Biswas_Sk}, we get
\begin{align}\label{caa4}
    \mathcal{D} \gtrsim [(w_{\la,\eps}-\hat{z}_{\la,\eps})^+]_{s,p}^p.
\end{align}
Combining, \eqref{caa1}--\eqref{caa4}, we get that $w_{\la,\eps}\leq \hat{z}_{\la,\eps}$ a.e. in $\Om$. Similarly, one has $\hat{z}_{\la,\eps}\leq \ov{z}_{\la',\eps}$. This proves the subclaim.

Hence, by \eqref{trunc}, $\hat{z}_{\la,\eps}$ satisfies,
\begin{equation}
\begin{split}
    \delp \hat{z}_{\la,\eps}&+\eps\frap \hat{z}_{\la,\eps}=\la \hat{z}_{\la,\eps}^{q-1}+\hat{z}_{\la,\eps}^{\pst-1}\text{ in }\Om,\\
    &\hat{z}_{\la,\eps}\geq0\text{ in }\Om,\quad \hat{z}_{\la,\eps}=0\text{ in }\RR \setminus \Om.
    \end{split}
\end{equation}
That means $\hat{z}_{\la,\eps}$ is a positive solution of \eqref{main_PDE}. Moreover, as $\hat{z}_{\la,\eps}$ is a global minimizer of $\hat{I}_{\la,\eps}$, $$I_{\la,\eps}(\hat{z}_{\la,\eps})=\hat{I}_{\la,\eps}(\hat{z}_{\la,\eps})\leq \hat{I}_{\la,\eps}(w_{\la,\eps})<J_{\la,\eps}(w_{\la,\eps})<0,$$
where $J_{\la,\eps}$ is as defined in \eqref{J-la-eps}. Hence, the claim is proved.

Now we let $\la_n\uparrow\La_\eps$ as $n \ra \infty$. Therefore, for each $n \in \mathbb{N}$, we have
\begin{equation}\label{limit-eqn}\tag{$\PP_{\la_n,\eps}$}
\begin{split}
    \delp \hat{z}_{\la_n,\eps}+ & \eps\frap \hat{z}_{\la_n,\eps}=\la \hat{z}_{\la_n,\eps}^{q-1}+\hat{z}_{\la_n,\eps}^{\pst-1}\text{ in }\Om, \\ & \hat{z}_{\la_n,\eps}\geq0\text{ in }\Om, \quad \hat{z}_{\la_n,\eps}=0\text{ in }\RR \setminus \Om.
\end{split}
\end{equation}
 where
 \begin{align*}
    &I_{\la_n,\eps}(\hat{z}_{\la_n,\eps})=\frac1p\rho_{\eps}(\hat{z}_{\la_n,\eps})^p-\frac{\la_n}q\|\hat{z}_{\la_n,\eps}\|_q^q-\frac1{\pst}\|\hat{z}_{\la_n,\eps}\|_{\pst}^{\pst}<0, \text{ and }\\
    &\langle I_{\la_n,\eps}'(\hat{z}_{\la_n,\eps}),\phi\rangle=0,\quad \forall \,\phi\in X_0.
    \end{align*}
    Observe that
    \begin{align*}
    &\limsup_{n\to\infty}I_{\La_\eps,\eps}(\hat{z}_{\la_n,\eps})=\limsup_{n\to\infty}I_{\la_n,\eps}(\hat{z}_{\la_n,\eps})\leq0,\\
    &\lim_{n\to\infty}\langle I_{\La_\eps,\eps}'(\hat{z}_{\la_n,\eps}),\phi\rangle=\lim_{n\to\infty}\langle I_{\la_n,\eps}'(\hat{z}_{\la_n,\eps}),\phi\rangle=0.
    \end{align*}
    Thus (up to a subsequence) $\{\hat{z}_{\la_n,\eps}\}$ is a PS sequence of $I_{\La_\eps,\eps}$. By Lemma \ref{converg_limi}, there exists $\hat{z}_{\La_\eps,\eps}\in X_0$ such that $\hat{z}_{\la_n,\eps}\rightharpoonup \hat{z}_{\La_\eps,\eps}$ in $X_0$, $\hat{z}_{\la_n,\eps}\to \hat{z}_{\La_\eps,\eps}$ a.e. in $\Om$ and $\hat{z}_{\La_\eps,\eps}$ solves ($\PP_{\La_\eps,\eps}$). Since $\hat{z}_{\la_n,\eps}\geq w_{\la_n,\eps}$ for every $n$ and  $\la_n$ is increasing, by Lemma \ref{comparison-1}, $\hat{z}_{\la_n,\eps}\geq w_{\la_n,\eps}\geq w_{\la_1,\eps}$ for every $n$, and hence $\hat{z}_{\La_\eps,\eps}\geq w_{\la_1,\eps}>0$ a.e. in $\Om$. This completes the proof.
\end{proof}
\begin{proof}[{\bf Proof of Theorem \ref{ABC}}]
    Combining Proposition \ref{nonexistence} and Lemmas \ref{Laepsinfty}--\ref{Laeps}, proof of Theorem \ref{ABC} follows.
\end{proof}

\section{Multiplicity of positive solutions}\label{sec_multiplicity}

In this section, we prove the multiplicity of positive solutions to \eqref{main_PDE} under a certain range of $q$. First, we discuss the regularity of weak solutions of \eqref{main_PDE}.
\begin{proposition}\label{regular}
    Let $p\in(1, \infty)$, $q\in(1,p)$, $\eps\in(0,1]$, and $\la_0>0$ be arbitrarily chosen and fixed. Let $u\in X_0$ be a weak solution of \eqref{main_PDE} ($u$ implicitly depends on both $\la$ and $\eps$)  and let $0<\la\leq\la_0$. Then
    \begin{enumerate}
        \item[{\rm (a)}] It holds
        \begin{equation}\label{e5.02}
            \|u\|_{L^{\infty}(\Om)}\lesssim_{N,s,\la_0,\Om} \left(1+\int_{\Om}|u|^{\pst\beta_1}\dx\right)^{\frac1{\pst(\beta_1-1)}},
        \end{equation}
        where $\beta_1=\frac{\pst+p-1}{p}$.
        \item[{\rm (b)}] There exists $r_0\in(0,1)$, independent of $\eps$, such that if $\rho_{\eps}(u)\leq r_0$, then $$\|u\|_{L^{\infty}(\Om)}\leq C({N,s,\la_0,\Om}).$$
        \item[{\rm (c)}] As a consequence of {\rm (b)}, there exists $\alpha=\alpha(N,p,s)\in(0,1)$ such that $\nabla u\in C_{\text{loc}}^{0,\alpha}(\Om)$ and for every $\Om_0\Subset\Om$, $$\|\nabla u\|_{C^{0,\alpha}(\Om_0)}\leq C(N,s,\la_0,\Om).$$
    \end{enumerate}
\end{proposition}
\begin{proof}
(a) We follow the approach of \cite[Theorem 1.1]{SVWZ}. For $\beta>1$ and $T>1$, we define
    \begin{equation}\label{e5.1}
        \phi(t):=\begin{cases}
            -\beta T^{\beta-1}(t+T)+T^\beta,\quad &t\leq-T;\\
            |t|^\beta,\quad &|t|\leq T;\\
            \beta T^{\beta-1}(t-T)+T^\beta,\quad &t\geq T.
        \end{cases}
    \end{equation}
We calculate, in the almost everywhere sense,
\begin{align*}
  &\phi'(t)=\begin{cases}
        -\beta T^{\beta-1},&t\leq-T;\\
        -\beta(-t)^{\beta-1},&-T\leq t\leq 0 ;\\
        \beta t^{\beta-1}, & 0 \leq t \leq T;\\
        \beta T^{\beta-1},&t\geq T,
    \end{cases}\; \text{ and } \\ & \phi''(t)=\begin{cases}
        \beta(\beta-1)t^{\beta-2},&0<t<T;\\
        \beta(\beta-1)(-t)^{\beta-2},&-T<t<0;\\
        0,&|t|> T.
    \end{cases}
\end{align*}
Notice that $\phi(t)\leq|t|^\beta$ for all $t \in \R$. By the definition of $\phi$, it is enough to see this inequality for $|t|\geq T$. If $|t|\geq T$, then $|t|= cT$ for some $c\geq1$ and we have
\begin{align*}
 \phi(t)=\begin{cases}\beta T^{\beta-1}(cT-T)+T^\beta, &\mbox{  for  }\, t\geq T;\\
-\beta T^{\beta-1}(-cT+T)+T^\beta, &\mbox{  for  }\, t\leq-T. \end{cases}
\end{align*}
Then using the inequality $\beta(c-1)+1\leq c^\beta$ with $\beta>1$ and $c\geq1$,  $\phi(t)=(\beta(c-1)+1)T^\beta\leq c^\beta T^\beta=|t|^\beta.$ Further, we observe
    \begin{enumerate}
        \item \label{phi(a)}$|\phi'(t)|\leq \beta |t|^{\beta-1}$ and $t\phi'(t)\leq\beta\phi(t)$, for all $t \in \R$.
        \item\label{phi(b)} $\phi$ is a convex and Lipschitz function with the Lipschitz constant $\beta T^{\beta-1}$.
    \end{enumerate}
Since $\phi$ is Lipschitz with Lipschitz constant $\be T^{\be-1}$, for any $u\in X_0$,
\begin{align*}
    \rho_\eps(\phi(u))^p=\|\nabla \phi(u)\|_p^p+\eps[\phi(u)]_{s,p}^p\leq (\be T^{\be-1})^p\left(\|\nabla u\|_p^p+\eps[u]_{s,p}^p\right)<\infty.
\end{align*}
Thus, $\phi(u)\in X_0$. Using the convexity of $\phi$, \cite[Lemma~2.8]{BPV},
\begin{equation}\label{e5.2}
\frap\phi(u)\leq|\phi'(u)|^{p-2}\phi'(u)\frap u,\;\text{ a.e. in }\Om.
\end{equation}
Using the Sobolev embedding $W_0^{s,p}(\Om)\hookrightarrow L^{\psst}(\Om)$, $\psst=\frac{Np}{N-sp}$ and the following identity from \cite[Proposition~2.10]{BPV}: $$\AA(v,u)=2\int_{\Om}u\frap v\dx, \; \forall \, u,v \in W_0^{s,p}(\Om),$$ we estimate
\begin{equation}\label{e5.3}
\begin{aligned}
\|\phi(u)\|_{\psst}^p&\leq C[\phi(u)]_{s,p}^p=2C\int_{\Om}\phi(u)\frap\phi(u)\dx\\
&\leq2C\int_{\Om}|\phi'(u)|^{p-2}\phi'(u)\phi(u)\frap u\dx.
\end{aligned}
\end{equation}
Also, notice that
\begin{equation}\label{e5.4}
\begin{aligned}
&\int_{\Om}|\nabla u|^{p-2}\nabla u\cdot\nabla(\phi(u)\phi'(u)|\phi'(u)|^{p-2})\\
&=\int_{\Om}|\nabla u|^p(|\phi'(u)|^p+\phi(u)\phi''(u)|\phi'(u)|^{p-2})\\
&\quad+\int_{\Om}|\nabla u|^p(p-2)|\phi'(u)|^{p-2}\phi(u)\phi''(u)\\
&\geq\int_{\Om} |\nabla u|^p|\phi'(u)|^p\dx.
\end{aligned}
\end{equation}
Now taking $v=\phi(u)\phi'(u)|\phi'(u)|^{p-2}$ as a test function in \eqref{weak} and combining \eqref{e5.3}-\eqref{e5.4}, we get
\begin{equation}\label{e5.41}
    \int_{\Om} |\nabla u|^p|\phi'(u)|^p\dx\leq \int_{\Om}|\phi'(u)|^{p-2}\phi'(u)\phi(u)(\la|u|^{q-2}u+|u|^{\pst-2}u)\dx.
\end{equation}
Using \eqref{e5.41} and and the classical Sobolev inequality,
\begin{align*}
\|\phi(u)\|_{\pst}^p&\leq\frac1{S_0}\|\nabla(\phi(u))\|_p^p=\frac1{S_0}\int_{\Om} |\nabla u|^p|\phi'(u)|^p\dx\\
&\leq\frac1{S_0}\int_{\Om}|\phi'(u)|^{p-2}\phi'(u)\phi(u)\left(\la|u|^{q-2}u+|u|^{\pst-2}u\right)\dx\\
&\leq \frac{\la+1}{S_0}\int_{\Om}|\phi'(u)|^{p-1}\phi(u)\left(1+|u|^{\pst-1}\right)\dx.
\end{align*}
Using the facts $\phi(u)\leq|u|^\beta$, $|\phi'(u)|\leq\beta|u|^{\beta-1}$ and $u\phi'(u)\leq\beta\phi(u)$, we see that
\begin{equation}\label{e5.5}
\|\phi(u)\|_{\pst}^p\leq C_0\beta^{p-1}\int_{\Om}\bigg(|u|^{p\beta-p+1}+(\phi(u))^p|u|^{\pst-p}\bigg)\dx,
\end{equation}
where $C_0=\frac{\la_0+1}{S_0}$.
We now choose $\beta$ in \eqref{e5.5} to be $\beta_1:=\frac{\pst+p-1}{p}$. Let $R>0$ to be fixed later. Using $\phi(u)\leq|u|^{\beta_1}$, we get
\begin{align}
&\int_{\Om}(\phi(u))^p|u|^{\pst-p}\dx \no\\
&\leq\int_{\{|u|\leq R\}}\frac{(\phi(u))^p}{|u|^{p-1}}R^{\pst-1}\dx+\left(\int_{\{|u|\geq R\}}(\phi(u))^{\pst}\right)^{\frac{p}{\pst}}\left(\int_{\{|u|\geq R\}}|u|^{\pst}\dx\right)^{\frac{\pst-p}{\pst}}\nonumber\\
&\leq R^{\pst-1}\int_{\Om} |u|^{\pst}\dx+\left(\int_{\Om}(\phi(u))^{\pst}\dx\right)^{\frac{p}{\pst}}\left(\int_{\{|u|\geq R\}}|u|^{\pst}\dx\right)^{\frac{\pst-p}{\pst}}.\label{e5.6}
\end{align}
By the Monotone Convergence Theorem, we  choose $R$ large enough so that
\begin{equation}\label{e5.7}
\left(\int_{\{|u|\geq R\}}|u|^{\pst}\right)^{\frac{\pst-p}{\pst}}\leq\frac{1}{2C_0\beta_1^{p-1}}.
\end{equation}
Thus using \eqref{e5.6} and \eqref{e5.7} in \eqref{e5.5}, we obtain
\begin{equation}\label{e5.7.1}
\|\phi(u)\|_{\pst}^p\leq 2C_0\beta_1^{p-1}(R^{\pst-1}+1)\|u\|_{\pst}^{\pst}.
\end{equation}
Taking $T\to\infty$, \eqref{e5.7.1} yields $u\in L^{\pst\beta_1}(\Om)$. Suppose $\beta>\beta_1$. Using $\phi(u)\leq |u|^\beta$ in the right hand side of \eqref{e5.5} and taking $T\to\infty$, we have
$$\left(\int_{\Om}|u|^{\pst\beta}\dx\right)^{\frac{p}{\pst}}\leq C_0\beta^{p-1}\int_{\Om} \left( |u|^{p\beta-p+1}+|u|^{\pst+\beta p-p} \right)\dx.$$
Applying Young's inequality with conjugates $\frac{b}{a}$ and $\frac{b}{b-a}$, where $a=p\beta-p+1$ and $b=\pst+\beta p-p$, in the first integral of right hand side, we get
\begin{align*}
\left(\int_{\Om}|u|^{\pst\beta}\dx\right)^{\frac{p}{\pst}}&\leq C_0\beta^{p-1}\left(\left(\frac{a}{b}\int_{\Om}|u|^{\pst+\beta p-p}+\frac{b-a}{b}|\Om|\right)+\int_{\Om}|u|^{\pst+\beta p-p}\dx\right)\\
&\leq 2C_0\beta^{p-1}(|\Om|+1)\left(1+\int_{\Om}|u|^{\pst+\beta p-p}\dx\right).
\end{align*}
Using the formula $(a+b)^p\leq 2^{p-1}(a^p+b^p)$, we see that    $$\left(1+\int_{\Om}|u|^{\pst\beta}\dx\right)^p\leq2^{p-1}+2^{p-1}\left(2C_0\beta^{p-1}(|\Om|+1)\left(1+\int_{\Om}|u|^{\pst+\beta p-p}\dx\right)\right)^{\pst}.$$
Thus,
\begin{equation}\label{e5.8}
    \left(1+\int_{\Om}|u|^{\pst\beta}\dx\right)^{\frac{1}{\pst(\beta-1)}}\leq (C\beta^{p-1})^{\frac{1}{p(\beta-1)}}\left(1+\int_{\Om}|u|^{\pst+\beta p-p}\dx\right)^{\frac{1}{p(\beta-1)}}.
\end{equation}
where $C$ is dependent only on $N,s,\la_0,\Om$ and independent of $\beta$.
For $m\geq1$, define $\beta_{m+1}$ such that $$\pst+p\beta_{m+1}-p=\pst\beta_m.$$
Thus $$\beta_{m+1}-1=\left(\frac{\pst}{p}\right)^m(\beta_1-1).$$
Define $C_{m}:=C\beta_{m}^{p-1}$ and
$A_{m}:=\left(1+\int_{\Om}|u|^{\pst\beta_{m}}\dx\right)^{\frac{1}{\pst(\beta_{m}-1)}},$
so that \eqref{e5.8} becomes
\begin{align}\label{interpol-1}
    A_{m+1}\leq C_{m+1}^{\frac{1}{p(\beta_{m+1}-1)}}A_m\leq\prod_{k=2}^{m+1}C_{k}^{\frac{1}{p(\beta_{k}-1)}}A_1.
\end{align}
Notice that $\beta_{m+1}=\left(\frac{\pst}{p}\right)^{m}(\beta_1-1)+1\leq 2\left(\frac{\pst}{p}\right)^{m+1}$. In general, $C_k=C\beta_k^{p-1}\leq 2^{p-1}C\left(\frac{\pst}{p}\right)^{k(p-1)}$. Further,
\begin{align*}
        \prod_{k=2}^{m+1}C_{k}^{\frac{1}{p(\beta_{k}-1)}}&\leq\prod_{k=2}^{m+1}\left(2^{p-1}C\left(\frac{\pst}{p}\right)^{k(p-1)}\right)^{\frac{\left(\tfrac{p}{\pst}\right)^{k-1}}{p(\beta_1-1)}}\\
        &=\left(2^{p-1}C\right)^{\frac1{p(\beta_1-1)}\sum_{k=2}^{m+1}\left(\frac{p}{\pst}\right)^{k-1}}\left(\frac{\pst}{p}\right)^{\frac{p-1}{p(\beta_1-1)}\sum_{k=2}^{m+1}k\left(\frac{p}{\pst}\right)^{k-1}}\\
        &\leq \left(2^{p-1}C\right)^{\frac1{p(\beta_1-1)}\sum_{k=2}^{\infty}\left(\frac{p}{\pst}\right)^{k-1}}\left(\frac{\pst}{p}\right)^{\frac{p-1}{p(\beta_1-1)}\sum_{k=2}^{\infty}k\left(\frac{p}{\pst}\right)^{k-1}}.
\end{align*}
Since $\tfrac{p}{\pst}<1$, both series are convergent and thus $A_m\leq C A_1,$ where $C$ is independent of $m$. Finally, we take the limit as $m \ra \infty$ in \eqref{interpol-1} and use interpolation to get \eqref{e5.02}.

\noi (b) To show (b), we return to \eqref{e5.6}: choosing $R=1$, using the Sobolev inequality and the fact that $\|\nabla u\|_{p}\leq \rho_{\eps}(u)\leq r_0$, we get
\begin{align*}
\int_{\Om}(\phi(u))^p|u|^{p^*-p}\dx&\leq \int_{\Om}|u|^{\pst}\dx+\left(\int_{\Om}(\phi(u))^{\pst}\dx\right)^{\frac{p}{\pst}}\left(\int_{\{|u|\geq 1\}}|u|^{\pst}\dx\right)^{\frac{\pst-p}{\pst}}\\
&\leq \int_{\Om}|u|^{\pst}\dx+\left(\int_{\Om}(\phi(u))^{\pst}\dx\right)^{\frac{p}{\pst}}\left(\int_{\Om}|u|^{\pst}\dx\right)^{\frac{\pst-p}{\pst}}\\
&\leq \int_{\Om}|u|^{\pst}\dx+\left(\int_{\Om}(\phi(u))^{\pst}\dx\right)^{\frac{p}{\pst}}S_0^{\frac{p-p^*}{p}}r_0^{p^*-p}
\end{align*}
Choosing $r_0>0$ such that
\begin{equation*}
S_0^{\frac{p-p^*}{p}}r_0^{p^*-p}\leq \frac{1}{2C_0\be_1^{p-1}}.
\end{equation*}
Proceeding as before and taking $T\to\infty$, we see that $A_1\leq C$ is independent of $\eps$. Hence, $\norm{u}_{\infty}\leq C$ where $C$ is independent of $\eps$.\\
\noi (c) As a consequence of (b), for any $d>N$, $L^d(\Om)$ and $L^N(\Om)$ norms of $\la |u|^{q-2}u+|u|^{\pst-2}u$ are uniformly bounded w.r.t. $\eps$. Now, (c) is just an application of \cite[Theorem 5]{DM} with the kernel $$K(x,y)=\frac{\eps}{|x-y|^{N+sp}},$$ (choosing $\texttt{k}=\eps$ and $\La=1$ in (1.7) of \cite{DM}).
\end{proof}

\begin{remark}\label{lower_remark}
Since in Proposition \ref{regular}, we have chosen $\la_0$ arbitrarily and then fixed it, in view of  Proposition \ref{B_r lemma}, we can choose $\la_0=\la^\#$  and $r_0>0$ so that for every $0<\la<\la^\#$ we have $u_{\la,\eps}$ lies in the interior of $B_{r_0}$ and  $\|u_{\la,\eps}\|_{\infty}\leq C$ where $C$ is independent of $\eps$.
\end{remark}

The following lemma will be useful for the existence of the second positive solution.
\begin{lemma}\label{lower_bound}
Let $p\in(1, \infty)$, $q\in(1,p)$, and $\eps\in(0,1]$. Let $\la>0$ be fixed and $w_{\la,\eps}$ be the unique positive solution of \eqref{sublinear}. Then there exists $\eps_0>0$ such that given any $B_R(x_0)\subset\Om$ and $0<r\leq\min\{1,R\}$, for every $0<\eps<\eps_0$,
        \begin{equation}
            w_{\la,\eps}\geq C, \text{ in } B_r(x_0),
        \end{equation}
where $C>0$ is independent of $\eps$.
\end{lemma}
\begin{proof}
Since $w_{\la,\eps}$ is a weak solution of \eqref{sublinear}, by Sobolev inequality,
$$\rho_\eps(w_{\la,\eps})^p=\la\|w_{\la,\eps}\|_q^q\leq \la C(N,s,q)\|\nabla w_{\la,\eps}\|_p^q\leq \la C(N,s,q)\rho_\eps(w_{\la,\eps})^q.$$
Thus, using $q<p$, $w_{\la,\eps}$ is bounded in $X_0$. Hence, up to a subsequence, there exists $w_\la\in X_0$ such that as $\eps\to0$, $w_{\la,\eps}\rightharpoonup w_\la$ in $X_0$ and $w_{\la,\eps}\to w_\la$ in $L^t(\Om)$ for every $1\leq t<\pst$. For $0<t<1$, by applying H\"{o}lder's inequality with the exponents $(\frac1t,\frac1{1-t})$, we get $w_{\la,\eps}\to w_\la$ in $L^t(\Om)$ as well. Let $J_{\la,0}$ be the functional associated with the purely local problem
\begin{equation}\label{loc_sub}
    \delp u=\la|u|^{q-2}u\text{ in }\Omega,\quad u>0 \text{ in }\Omega,\quad
    u=0\text{ in }\RR\setminus\Omega.
\end{equation}
Then, using the act that $w_{\la,\eps}$ is a global minimizer of $J_{\la,\eps}$, for any $\phi\in X_0$,
\begin{align*}
    J_{\la,0}(w_\la)=\frac1p\|\nabla w_\la\|_p^p-\frac\la q\|w_\la\|_q^q\leq \liminf_{\eps\to0}J_{\la,\eps}(w_{\la,\eps})\leq \liminf_{\eps\to0}J_{\la,\eps}(\phi)=J_{\la,0}(\phi).
\end{align*}
Thus $w_{\la}$ is a global minimizer for $J_{\la,0}$. By \cite{DS}, $J_{\la,0}$ has a unique global minimizer which is also the unique solution of \eqref{loc_sub}. Thus $w_{\la}$ is the unique solution of \eqref{loc_sub}. Further, since $w_{\la,\eps}$ is a weak solution of \eqref{sublinear}, it is also a weak supersolution of the problem
\begin{equation*}
    \delp u+\eps\frap u=0\text{ in }\Omega,\quad
    u=0\text{ in }\RR\setminus\Omega.
\end{equation*}
Thus by the weak Harnack inequality \cite[Theorem 3.5]{PG} and H\"{o}lder's inequality, for any $x\in B_r(x_0)$, some $0<Q<\pst$ and some $C>0$ independent of $\eps$,
$$w_{\la,\eps}(x)\geq C\left(\fint_{B_r(x_0)} w_{\la,\eps}^Q\dx\right)^Q.$$
As $w_{\la,\eps}\to w_\la$ in $L^t(\Om)$ for every $0< t<\pst$, $$\lim_{\eps\to0}\left(\fint_{B_r(x_0)} w_{\la,\eps}^Q\dx\right)^Q=\left(\fint_{B_r(x_0)} w_{\la}^Q\dx\right)^Q.$$
Further, since $w_{\la}>0$ a.e. in $\Om$, there exists $\eps_0>0$ such that, for any $0<\eps<\eps_0$, $w_{\la,\eps}\geq C \text{ a.e. in }B_r(x_0).$
\end{proof}

\begin{remark}
Since $u_{\la,\eps}$ is a weak positive solution of \eqref{main_PDE}, it is a weak supersolution of \eqref{sublinear}. By Lemma \ref{comparison-1}, $u_{\la,\eps}\geq w_{\la,\eps} \text{ a.e. in }\Om$, and hence using Lemma \ref{lower_bound}, there exists $\eps_0>0$ such that given any $B_R(x_0)\subset\Om$ and $0<r\leq\min\{1,R\}$, for every $0<\eps<\eps_0$, we have
        \begin{equation}\label{eq_lower}
            u_{\la,\eps}\geq C\quad\text{on } B_r(x_0),
        \end{equation}
where $C>0$ is independent of $\eps$.
\end{remark}
\begin{proposition}\label{energy_prop}
Let $p \in [2, \infty)$ and $(p,q)$ satisfy the condition \eqref{A_pq}. Let $\la\in(0,\la^\#)$, and $u_{\la,\eps}$ be as in the Proposition \ref{B_r lemma}. Let $\eps_0$ be given as in Lemma \ref{lower_bound}. Then there exist $\eps_\la\in(0,\eps_0)$, $R_0>0$ and a positive function $\Psi\in X_0$ such that for every $\eps\in(0,\eps_\la)$ and $R\geq R_0$,
\begin{align}
    &I_{\la,\eps}(u_{\la,\eps}+R\Psi)<I_{\la,\eps}(u_{\la,\eps}), \label{energy estimate}\\
    &I_{\la,\eps}(u_{\la,\eps}+tR_0\Psi)<I_{\la,\eps}(u_{\la,\eps})+\frac1N S_0^{\frac{N}{p}},\quad\forall\,t\in[0,1].\label{energy estimate-2}
\end{align}
\end{proposition}
\begin{proof}
From Remark \ref{lower_remark}, we know that $u_{\la,\eps}\leq C$.
 Now we choose a Lebesgue point $y$ of $u_{\la,\eps}$ in $\Om$. Let $r>0$ be such that $B_{2r}(y)\Subset\Om$. Choose a cutoff function $\phi\in C_c^{\infty}(B_{2r}(y))$ such that $0\leq\phi\leq1$ in $\Om$, $\phi\equiv1$ in $B_r(y)$ and $|\nabla\phi|\leq2/r.$ We then consider the family of scalings of the Talenti function:
    $$V_{\eps}(x):=K_{N,p}\frac{\eps^{\frac{\al(N-p)}{p(p-1)}}}{\left(\eps^{\frac{\al p}{p-1}}+|x-y|^{\frac{p}{p-1}}\right)^{\frac{N-p}{p}}},\quad \eps>0,$$
    where $K_{N,p}$ is a normalization constant and $\alpha>0$ is to be chosen later. Since $V_\eps\in \dot{W}^{1,p}(\RR),$ we consider the following family of functions $U_\eps:=V_\eps\phi$ which is supported in $\Om$. Using \cite[pg. 947]{GaPa1} and \cite[Lemma 5.3]{DFB}, as $\eps\to0$,
    \begin{equation}\label{e7.2}
        \begin{aligned}
       &\|\nabla U_\eps\|_p^p= K_1+O\left(\eps^{\frac{\alpha(N-p)}{p-1}}\right), \\
        & \|U_\eps\|_{\pst}^{\pst}= K_2-O\left(\eps^{\frac{\al N }{p-1}}\right), \text{ where } \frac{K_1}{K_2^{\frac{p}{\pst}}}=S_0, \text{ and }\\
        &[U_\eps]_{s,p}^p=O\left(\eps^{\frac{\al (N-p)}{p-1}}\right)+O\left(\eps^{\alpha(p-ps)}\right).
    \end{aligned}
    \end{equation}
    Moreover, we use the following estimates (see \cite[Lemma A5]{GaPa1}):
    \begin{equation}\label{i-3}
        \int_{\Om}|\nabla U_\eps|^t\dx\leq C\eps^{\alpha\frac{N-p}{p(p-1)}t}, \text{ when } 1\leq t<\frac{N(p-1)}{N-1}.
    \end{equation}
    Further, for $t>0$ and $0< \eps< r^{\frac{1}{\al}}$,
    \begin{equation}\label{lb-2}
    \begin{split}
\int_{B_r(y)}V_\eps^t\dx &=\frac{\eps^{\alpha\frac{N-p}{p(p-1)}t}}{\eps^{\alpha\frac{N-p}{p-1}t}}\int_{B_r(y)}\frac{1}{\left(1+\left(\frac{|x-y|}{\eps^\alpha}\right)^{\frac{p}{p-1}}\right)^{\frac{N-p}{p}t}}\dx\\
&=C\eps^{-\alpha\frac{N-p}{p}t}\int_{0}^r\frac{\tau^{N-1}}{\left(1+(\frac{\tau}{\eps^\alpha})^{\frac{p}{p-1}}\right)^{\frac{N-p}{p}t}}\mathrm{d}\tau\\&=C\eps^{\alpha\left(N-\frac{N-p}{p}t\right)}\int_0^{\frac{r}{\eps^{\alpha}}}\frac{s^{N-1}}{(1+s^{\frac{p}{p-1}})^{\frac{N-p}{p}t}}\ds \\
&\geq C\eps^{\alpha\left(N-\frac{N-p}{p}t\right)}\int_0^{1}\frac{s^{N-1}}{(1+s^{\frac{p}{p-1}})^{\frac{N-p}{p}t}}\ds=C\eps^{\alpha\left(N-\frac{N-p}{p}t\right)}.
\end{split}
\end{equation}
  For $0 \le t \le 1 \le R$, we set $w=u_{\la,\eps}+tRU_{\eps}$. Then we have
    \begin{align}\label{(a)}
        I_{\la,\eps}(w)=\frac{\|\nabla u_{\la,\eps}+tR\nabla U_{\eps}\|_p^p+\eps[u_{\la,\eps}+tRU_{\eps}]_{s,p}^p}{p}-\frac{\|w\|_{\pst}^{\pst}}{\pst}-\frac{\la\|w\|_q^q}{q}.
    \end{align}

Now, depending on the range of $p$, we divide the rest of the proof into two parts.
\medskip

\noi \textbf{When $2\leq p<3$:}
Using the inequality \eqref{i-1}, we get the following gradient estimate:
\begin{equation}\label{GE1}
\begin{aligned}
    \|\nabla u_{\la,\eps}+tR\nabla U_{\eps}\|_p^p&\leq \|\nabla u_{\la,\eps}\|_p^p+(tR)^p\|\nabla U_\eps\|_p^p+ptR\int_{\Om}|\nabla u_{\la,\eps}|^{p-2}\nabla u_{\la,\eps}\cdot\nabla U_{\eps}\dx\\&\quad+C(tR)^{\zeta_1}\int_{\Om}|\nabla u_{\la,\eps}|^{p-\zeta_1}|\nabla U_\eps|^{\zeta_1}\dx,
\end{aligned}
\end{equation}
where $\zeta_1\in[p-1,2]$. For the Gagliardo norm, we again use \eqref{i-1}. For brevity, we denote, for $x,y\in\RR$,
$$\d\mu :=\frac{\dx\dy}{|x-y|^{N+sp}},\; f(x,y):=u_{\la,\eps}(x)-u_{\la,\eps}(y),\;  g(x,y):=tR(U_{\eps}(x)-U_\eps(y)).$$
We write
\begin{align*}
    [u_{\la,\eps}+tRU_{\eps}]_{s,p}^p&=\iint\limits_{\R^{2N}}\left|(u_{\la,\eps}(x)-u_{\la,\eps}(y))+tR(U_{\eps}(x)-U_\eps(y))\right|^p\d\mu\\
&=\iint\limits_{U}|f+g|^p\d\mu+\iint\limits_{U^c}|g|^p\d\mu=:I_1+I_2,
\end{align*}
where $U=\{(x,y) \in \R^{2N}: f(x,y) \neq 0\}$.
Now, we estimate $I_1$ as
\begin{align*}
     I_1&=\iint\limits_{U} |(f+g)^2|^{\frac p2}\d\mu=\iint\limits_{U}|f|^p\left|1+\frac{|g|^2}{|f|^2}+2\frac{g}{f}\right|^{\frac p2}\d\mu\\
    &=\iint\limits_{U} \left( |f|^p\left|1+\frac{|g|^2}{|f|^2}+2\frac{|g|}{|f|}\,\text{sgn}\left(\frac gf\right)\right|^{\frac p2} \right) \d\mu\\
    &\leq \iint\limits_{U} \left(|f|^p+|g|^p+p|f|^{p-1}|g|\,\text{sgn}\left(\frac gf\right)+C|f|^{p-\zeta_1}|g|^{\zeta_1} \right)\d\mu,
\end{align*}
where $\zeta_1\in[p-1,2]$ and the last inequality follows from \eqref{i-1}. We note that $$\iint\limits_{U}|f|^p\d\mu=\iint\limits_{\R^{2N}}|f|^p\d\mu=[u_{\la,\eps}]_{s,p}^p, \text{ and  } \iint\limits_{U}|g|^p\d\mu + I_2 = [tRU_\eps]_{s,p}^p.$$
Since, $\text{sgn}\left(\frac{g}{f}\right)=\text{sgn}(fg)$ on $U$, we get
\begin{align*}
     \iint\limits_{U}|f|^{p-1}|g| \, \text{sgn}\left(\frac gf\right)\d\mu&=\iint\limits_{U}|f|^{p-2}|f||g| \, \text{sgn}(fg)\d\mu
    \\&=\iint\limits_{U}|f|^{p-2}fg\d\mu=\AA(u_{\la,\eps},tRU_\eps).
\end{align*}
Applying H\"{o}lder's inequality with coefficients $\left(\frac{p}{p-\zeta_1},\frac{p}{\zeta_1}\right)$, we have $$\iint\limits_{U}C|f|^{p-\zeta_1}|g|^{\zeta_1}\d\mu=\iint\limits_{\R^{2N}}C|f|^{p-\zeta_1}|g|^{\zeta_1}\d\mu\leq C[u_{\la,\eps}]_{s,p}^{p-\zeta_1}[tRU_{\eps}]_{s,p}^{\zeta_1}.$$
Combing the above inequalities, we finally get
\begin{align}\label{(c)}
    [u_{\la,\eps}+tRU_{\eps}]_{s,p}^p
    &\leq [u_{\la,\eps}]_{s,p}^p+(tR)^p[U_{\eps}]_{s,p}^p+ptR\AA(u_{\la,\eps},U_\eps)\no\\
    &+C(tR)^{\zeta_1}[u_{\la,\eps}]_{s,p}^{p-\zeta_1}[U_{\eps}]_{s,p}^{\zeta_1}.
\end{align}

Using \eqref{GE1} and \eqref{(c)} in \eqref{(a)}, we obtain
\begin{align*}
I_{\la,\eps}(w)&\leq\frac1p\rho_\eps(u_{\la,\eps})^p+\frac{(tR)^p}p\rho_\eps(U_\eps)^p\\
&\quad+tR\left(\int_{\Om}|\nabla u_{\la,\eps}|^{p-2}\nabla u_{\la,\eps}\cdot\nabla U_{\eps}\dx+\eps\AA(u_{\la,\eps},U_\eps)\right)\\&\quad+C(tR)^{\zeta_1}\left(\int_{\Om}|\nabla u_{\la,\eps}|^{p-\zeta_1}|\nabla U_\eps|^{\zeta_1}\dx+\eps[u_{\la,\eps}]_{s,p}^{p-\zeta_1}[U_{\eps}]_{s,p}^{\zeta_1}\right)\\&\quad-\frac{\|w\|_{\pst}^{\pst}}{\pst}-\frac{\la\|w\|_q^q}{q}.
\end{align*}
Since $u_{\la,\eps}$ solves \eqref{main_PDE}, we further have
\begin{align*}
I_{\la,\eps}(w)&\leq\frac1p\rho_\eps(u_{\la,\eps})^p+\frac{(tR)^p}p\rho_\eps(U_\eps)^p\\
&\quad+C(tR)^{\zeta_1}\left(\int_{\Om}|\nabla u_{\la,\eps}|^{p-\zeta_1}|\nabla U_\eps|^{\zeta_1}\dx+\eps[u_{\la,\eps}]_{s,p}^{p-\zeta_1}[U_{\eps}]_{s,p}^{\zeta_1}\right)\\&\quad-\frac1{\pst}\left(\|u_{\la,\eps}+tRU_{\eps}\|_{\pst}^{\pst}-\pst tR\int_{\Om}u_{\la,\eps}^{\pst-1}U_\eps\dx\right)\\&\quad-\frac{\la}q\left(\|u_{\la,\eps}+tRU_{\eps}\|_q^q-qtR\int_{\Om}u_{\la,\eps}^{q-1}U_\eps\dx\right)\\
&=I_{\la,\eps}(u_{\la,\eps})+\frac{(tR)^p}p\rho_\eps(U_\eps)^p-\frac{(tR)^{\pst}}{\pst}\|U_\eps\|_{\pst}^{\pst}-(tR)^{\pst-1}\int_{\Om}U_\eps^{\pst-1}u_{\la,\eps}\dx\\
&\quad+L_1-L_2-L_3,
\end{align*}
where $$L_1:=C(tR)^{\zeta_1}\left(\int_{\Om}|\nabla u_{\la,\eps}|^{p-\zeta_1}|\nabla U_\eps|^{\zeta_1}\dx+\eps[u_{\la,\eps}]_{s,p}^{p-\zeta_1}[U_{\eps}]_{s,p}^{\zeta_1}\right),$$
and
\begin{align*}
 L_2&:=\frac1{\pst}\int_{\Om}|u_{\la,\eps}+tRU_{\eps}|^{\pst}-|u_{\la,\eps}|^{\pst}-(tR)^{\pst}|U_{\eps}|^{\pst}\\
 &-\pst tRu_{\la,\eps}U_\eps\left(u_{\la,\eps}^{\pst-2}+\left(tRU_\eps\right)^{\pst-2}\right)\dx,
\end{align*}
and $$L_3:=\frac{\la}q\int_{\Om}|u_{\la,\eps}+tRU_{\eps}|^q-|u_{\la,\eps}|^q-qtRu_{\la,\eps}^{q-1}U_\eps\dx.$$
Using the convexity of the map $f(t)=t^q$, we see that $L_3\geq0$.
Next, we estimate $L_1$.  Because of Lemma \ref{regular}, the family $\{u_{\la,\eps}\}_\eps$ is uniformly bounded (with respect to $\eps$) in $L^{\infty}(\Om)$, $W^{1,p}(\Om)$, and by Poincar\'{e} inequality and Sobolev embedding theorem, in $W^{s,p}(\RR)$ as well. We now fix
$\zeta_1$ satisfying
\begin{align*}
    p-1 < \zeta_1 < \min \left\{ 2, \frac{N(p-1)}{N-1}  \right\}.
\end{align*}
Then using Lemma \ref{regular}-(c) and \eqref{i-3} for the gradient term and \eqref{e7.2} in the Gagliardo norms, we obtain
\begin{align*}
    L_1&\leq C(tR)^{\zeta_1}\left(\int_{B_{2r}(y)}|\nabla u_{\la,\eps}|^{p-\zeta_1}|\nabla U_\eps|^{\zeta_1}\dx+\eps[u_{\la,\eps}]_{s,p}^{p-\zeta_1}[U_\eps]_{s,p}^{\zeta_1}\right)\\
    &\leq C(tR)^{\zeta_1}\left(O\left(\eps^{\alpha\zeta}\right)+O\left(\eps^{1+\alpha\frac{N-p}{p}\frac{\zeta_1}{p-1}}\right)+O\left(\eps^{1+\alpha(1-s)\zeta_1}\right)\right)\\
    &\leq C(tR)^{\zeta_1}\left(O\left(\eps^{\alpha\zeta}\right)+O\left(\eps^{1+\alpha\frac{N-p}{p}}\right)+O\left(\eps^{1+\alpha(p-1)(1-s)}\right)\right),
\end{align*}
where $\zeta := \frac{\zeta_1(N-p)}{p(p-1)}>\frac{N-p}{p}$.
We next claim that
\begin{equation}\label{L2-estimate}
    \begin{cases}
        L_2\geq0,\quad\text{if } \pst\geq3;\\
        |L_2|\leq C(tR)^\beta \eps^{\alpha\frac{N}{p}\theta},\text{ for some }\beta>0,\; \forall \, 0<\theta<1, \quad\text{if } 2<\pst<3.
    \end{cases}
\end{equation}
Indeed, using Lemma \ref{q_ineq}-(iii), we  see that $L_2\geq0$ for any $\pst\geq3$. Now for $2<\pst<3$, we make use of the inequality \eqref{elemt.ineq_pst}.
For $i=1,2$, let $a_i+b_i=\pst-1$ and $a_i,b_i>0$ to be chosen later. We use the uniform boundedness of $u_{\la,\eps}$ in $L^{\infty}(\Om)$ w.r.t. $\eps$. Then
\begin{align*}
    |L_2|&\leq C\int_{\{u_{\la,\eps}\leq tRU_\eps\}}u_{\la,\eps}^{\pst-1}tRU_\eps\dx+C\int_{\{tRU_\eps\leq u_{\la,\eps} \}}u_{\la,\eps}(tRU_\eps)^{\pst-1}\dx\\
    &\leq C\int_{\{u_{\la,\eps}\leq tRU_\eps\}} u_{\la,\eps}^{a_1}(tRU_\eps)^{1+b_1}\dx+C\int_{\{tRU_\eps\leq u_{\la,\eps} \}}u_{\la,\eps}^{1+a_2}(tRU_\eps)^{b_2}\dx\\
    &\leq C(tR)^{1+b_1}\int_{B_{2r}(y)}U_\eps^{1+b_1}\dx+C(tR)^{b_2}\int_{B_{2r}(y)}U_{\eps}^{b_2}\dx\\
    &\leq C\left(tR\eps^{\alpha\frac{N-p}{p(p-1)}}\right)^{1+b_1}\int_{B_{2r}(0)}\frac{1}{|x|^{\frac{N-p}{p-1}(1+b_1)}}\dx\\&\quad+C\left(tR\eps^{\alpha\frac{N-p}{p(p-1)}}\right)^{b_2}\int_{B_{2r}(0)}\frac{1}{|x|^{\frac{N-p}{p-1}b_2}}\dx.
\end{align*}
Note that for $0<1+b_1,\, b_2<\frac{N(p-1)}{N-p}$ the above integrals are finite. This leads us to choose $b_1,b_2$ with $p-1<1+b_1,\, b_2<\frac{N(p-1)}{(N-p)}$ such that for $\theta\in(0,1)$ and $0<\beta<\frac{N(p-1)}{N-p}$,
$$|L_2|\leq C(tR)^\beta\eps^{\alpha\frac{N}{p}\theta},$$
where $C>0$ is independent of $\eps$.
Thus, \eqref{L2-estimate} holds. Using \eqref{eq_lower} and \eqref{lb-2} we get
\begin{align}\label{estimate-3}
    \int_{\Om}U_{\eps}^{\pst-1}u_{\la,\eps}\dx\geq C\int_{B_r(y)}V_\eps^{\pst-1}\dx\geq C\eps^{\alpha\frac{N-p}p}.
\end{align} Choose $\alpha$ small enough such that
\begin{align}\label{choice-1}
    1+\alpha p(1-s)>1+\alpha(p-1)(1-s)>\alpha \frac{N-p}{p-1}.
\end{align}
Combining all the above inequalities,
 \begin{equation}\label{(d1)}
     \begin{aligned}
         I_{\la,\eps}(w)&\leq I_{\la,\eps}(u_{\la,\eps})+\frac{(tR)^p}p K_1-\frac{(tR)^{\pst}}{\pst}K_2+\frac{(tR)^p}pO\left(\eps^{\alpha\frac{N-p}{p-1}}\right)\\&\quad +\frac{(tR)^{\pst}}{\pst}O\left(\eps^{\alpha \frac{N}{p-1}}\right)-(tR)^{\pst-1} C\eps^{\alpha \frac{N-p}p}+C(tR)^\beta\eps^{\alpha\frac{N}{p}\theta}\\&\quad+C(tR)^{\zeta_1}\left(O\left(\eps^{\alpha\zeta}\right)+O\left(\eps^{1+\alpha\frac{N-p}{p}}\right)+O\left(\eps^{\alpha\frac{N-p}{p-1}}\right)\right).
     \end{aligned}
 \end{equation}
For $t=1$, we choose $R_0>0$ such that for any $R\geq R_0$ and $\eps$ small enough,
\begin{equation}\label{(e1)}
    \begin{aligned}
    &\left(\frac{R_0^p}p K_1-\frac{R_0^{\pst}}{\pst}K_2-R_0^{\pst-1}C\eps^{\alpha \frac{N-p}p}\right)\\&+\Bigg(C R_0^{\zeta_1}\left(O\left(\eps^{\alpha\zeta}\right)+O\left(\eps^{1+\alpha\frac{N-p}{p}}\right)+O\left(\eps^{\alpha\frac{N-p}{p-1}}\right)\right)\\
    &+\frac{R_0^p}pO\left(\eps^{\alpha\frac{N-p}{p-1}}\right)+\frac{R_0^{\pst}}{\pst}O\left(\eps^{\alpha \frac{N}{p-1}}\right)+CR_0^\beta\eps^{\alpha\frac{N}{p}\theta}\Bigg)<0.
\end{aligned}
\end{equation}
Thus $I_{\la,\eps}(u_{\la,\eps}+RU_\eps)<I_{\la,\eps}(u_{\la,\eps})$, for $R\geq R_0$, proving  \eqref{energy estimate}. Next, we fix $R=R_0$ and define $$\phi(t)=\frac{(tR_0)^p}p K_1-\frac{(tR_0)^{\pst}}{\pst}K_2-(tR_0)^{\pst-1}C\eps^{\alpha \frac{N-p}p}.$$Recall that $\zeta>\frac{N-p}{p}$ and the choice of $b_1,b_2$ implies that  $\theta>\frac{N-p}{p}$. Then except $-(tR)^{\pst-1} C\eps^{\alpha \frac{N-p}p}$, all other epsilons in \eqref{(d1)} have exponents larger than $\al \frac{N-p}p$. Hence for $t\in[0,1]$, we rewrite \eqref{(d1)} as
\begin{equation}
    I_{\la,\eps}(u_{\la,\eps}+tR_0U_{\eps})\leq I_{\la,\eps}(u_{\la,\eps})+\phi(t)+o(\eps^{\al \frac{N-p}{p}}).
\end{equation}
Notice that $\phi(0)=0$, $\phi(t)>0$ for small $t>0$ and by \eqref{(e1)}, $\phi(1)<0$. So $\phi$ attains its maximum at some $t_\eps\in(0,1)$. Then,
\begin{equation}\label{(f1)}
    I_{\la,\eps}(u_{\la,\eps}+tR_0U_{\eps})\leq I_{\la,\eps}(u_{\la,\eps})+\phi(t_\eps)+o\left(\eps^{\alpha\frac{N-p}{p}}\right).
\end{equation}
Suppose $t_\eps\to0$ as $\eps\to0$. Then for small enough $\eps$, from \eqref{(f1)} we get
$$I_{\la,\eps}(u_{\la,\eps}+tR_0U_{\eps})<I_{\la,\eps}(u_{\la,\eps})+\frac1N  S_0^{\frac{N}{p}}.$$
If $t_{\eps} \not \rightarrow 0$, then there exists some $T\in(0,1)$ such that $T<t_\eps<1$ for every $\eps$ small. Using the expression of $S_0$ in \eqref{e7.2}, We observe that $$\max_{t\geq0}\left(\frac{(tR_0)^p}p K_1-\frac{(tR_0)^{\pst}}{\pst}K_2\right)=\frac{(t_0R_0)^p}p K_1-\frac{(t_0R_0)^{\pst}}{\pst}K_2=\frac1N S_0^{\frac{N}{p}},$$ where $t_0^{\pst-p}=\frac{R_0^pK_1}{R_0^{\pst}K_2}$ is the maximum point. Thus \eqref{(f1)} gives
\begin{equation}\label{fin_est}
    I_{\la,\eps}(u_{\la,\eps}+tR_0U_{\eps})\leq I_{\la,\eps}(u_{\la,\eps})+\frac1N S_0^{\frac{N}{p}}-(TR_0)^{\pst-1}C\eps^{\alpha \frac{N-p}p}+o\left(\eps^{\alpha\frac{N-p}{p}}\right).
\end{equation}
Hence, there exists $\eps_0\in(0,1)$ such that for all $\eps<\eps_0$ and any $t\in[0,1]$, $$I_{\la,\eps}(u_{\la,\eps}+tR_0U_{\eps})< I_{\la,\eps}(u_{\la,\eps})+\frac1NS_0^{\frac{N}{p}}.$$
\noi \textbf{When $p\geq3$:} In this case, we use Lemma \ref{q_ineq}-(vi). Thus, we can write the gradient term as
\begin{equation}\label{(b)}
    \begin{aligned}
        &\|\nabla u_{\la,\eps}+tR\nabla U_{\eps}\|_p^p\leq \|\nabla u_{\la,\eps}\|_p^p+(tR)^p\|\nabla U_\eps\|_p^p+ptR\int_{\Om}|\nabla u_{\la,\eps}|^{p-2}\nabla u_{\la,\eps}\cdot\nabla U_{\eps}\dx\\&+C\left((tR)^2\int_{\Om}|\nabla u_{\la,\eps}|^{p-2}|\nabla U_{\eps}|^{2}\dx+(tR)^{p-1}\int_{\Om}|\nabla u_{\la,\eps}|\,|\nabla U_{\eps}|^{p-1}\dx\right).
    \end{aligned}
\end{equation}
Similar to the previous case, using \eqref{i-2} and the H\"{o}lder's inequality, we can write the nonlocal term as
\begin{equation}\label{faa1}
    \begin{aligned}
    [u_{\la,\eps}+tRU_\eps]_{s,p}^p&\leq [u_{\la,\eps}]_{s,p}^p+(tR)^p[U_\eps]_{s,p}^p+ptR\AA(u_{\la,\eps},U_\eps)\\&\quad+C((tR)^2[u_{\la,\eps}]_{s,p}^{p-2}[U_\eps]_{s,p}^2+(tR)^{p-1}[u_{\la,\eps}]_{s,p}[U_\eps]_{s,p}^{p-1}).
\end{aligned}
\end{equation}
As we will see later, for small enough $\alpha$, the last two cross terms in \eqref{faa1} have order greater than $\eps^{\alpha\frac{N-p}{p}}$. But the last two cross terms in \eqref{(b)} have order less than $\eps^{\al\frac{N-p}{p}}$.  This poses a challenge because, following the same approach as in the previous case, we obtain
\begin{align*}
 I_{\la,\eps}(u_{\la,\eps}+tR_0U_{\eps})&\leq I_{\la,\eps}(u_{\la,\eps})+\frac1N   S_0^{\frac{N}{p}}-(TR_0)^{\pst-1}C\eps^{\alpha \frac{N-p}p}\\&+R_0^{2}O\left(\eps^{\alpha\frac{2(N-p)}{p(p-1)}}\right)+o\left(\eps^{\alpha\frac{N-p}{p}}\right).
\end{align*}
Notice that, since $\frac{2(N-p)}{p(p-1)}<\frac{N-p}{p}$, for small enough $\eps,$ $\eps^{\frac{2(N-p)}{p(p-1)}}$ dominates $\eps^{\alpha\frac{N-p}{p}}$. So it is difficult to conclude \eqref{energy estimate-2} from the above estimate. To address this issue, we retain the $q$-terms in this case. Using \eqref{(b)} and \eqref{(c)} in \eqref{(a)}, we obtain
\begin{align*}
I_{\la,\eps}(w)&\leq\frac1p\rho_\eps(u_{\la,\eps})^p+\frac{(tR)^p}p\rho_\eps(U_\eps)^p\\&\quad+tR\left(\int_{\Om}|\nabla u_{\la,\eps}|^{p-2}\nabla u_{\la,\eps}\cdot\nabla U_{\eps}\dx+\eps\AA(u_{\la,\eps},U_\eps)\right)\\&\quad+C\left((tR)^2\int_{\Om}|\nabla u_{\la,\eps}|^{p-2}|\nabla U_{\eps}|^{2}\dx+(tR)^{p-1}\int_{\Om}|\nabla u_{\la,\eps}|\,|\nabla U_{\eps}|^{p-1}\dx\right)\\
&\quad+C((tR)^2\eps[u_{\la,\eps}]_{s,p}^{p-2}[U_\eps]_{s,p}^2+(tR)^{p-1}\eps[u_{\la,\eps}]_{s,p}[U_\eps]_{s,p}^{p-1})\\
&\quad-\frac{\|w\|_{\pst}^{\pst}}{\pst}-\frac{\la\|w\|_q^q}{q}.
\end{align*}
Since $u_{\la,\eps}$ solves \eqref{main_PDE}, we further have
\begin{align*}
I_{\la,\eps}(w)&\leq\frac1p\rho_\eps(u_{\la,\eps})^p+\frac{(tR)^p}p\rho_\eps(U_\eps)^p\\&\quad+C((tR)^2\eps[u_{\la,\eps}]_{s,p}^{p-2}[U_\eps]_{s,p}^2+(tR)^{p-1}\eps[u_{\la,\eps}]_{s,p}[U_\eps]_{s,p}^{p-1})\\&\quad-\frac1{\pst}\left(\|u_{\la,\eps}+tRU_{\eps}\|_{\pst}^{\pst}-\pst tR\int_{\Om}u_{\la,\eps}^{\pst-1}U_\eps\dx\right)\\&\quad-\frac{\la}q\left(\|u_{\la,\eps}+tRU_{\eps}\|_q^q-qtR\int_{\Om}u_{\la,\eps}^{q-1}U_\eps\dx\right)\\
&\quad+C\left((tR)^2\int_{\Om}|\nabla u_{\la,\eps}|^{p-2}|\nabla U_{\eps}|^{2}\dx+(tR)^{p-1}\int_{\Om}|\nabla u_{\la,\eps}|\,|\nabla U_{\eps}|^{p-1}\dx\right)\\
&=I_{\la,\eps}(u_{\la,\eps})+\frac{(tR)^p}p\rho_\eps(U_\eps)^p-\frac{(tR)^{\pst}}{\pst}\|U_\eps\|_{\pst}^{\pst}-(tR)^{\pst-1}\int_{\Om}U_\eps^{\pst-1}u_{\la,\eps}\dx\\&\quad+L_1-L_2-L_3,
\end{align*}
where
\begin{align*}
    L_1&:=C\left((tR)^2\int_{\Om}|\nabla u_{\la,\eps}|^{p-2}|\nabla U_{\eps}|^{2}\dx+(tR)^{p-1}\int_{\Om}|\nabla u_{\la,\eps}|\,|\nabla U_{\eps}|^{p-1}\dx\right)\\
    &\quad+C\left((tR)^2\eps[u_{\la,\eps}]_{s,p}^{p-2}[U_\eps]_{s,p}^2+(tR)^{p-1}\eps[u_{\la,\eps}]_{s,p}[U_\eps]_{s,p}^{p-1}\right),
\end{align*}
and
\begin{align*}
L_2 &:=\frac1{\pst}\int_{\Om}|u_{\la,\eps}+tRU_{\eps}|^{\pst}-|u_{\la,\eps}|^{\pst}-(tR)^{\pst}|U_{\eps}|^{\pst}\\
&\quad-\pst tRu_{\la,\eps}U_\eps(u_{\la,\eps}^{\pst-2}+(tRU_\eps)^{\pst-2})\dx,
\end{align*}
and $$L_3:=\frac{\la}q\int_{\Om}|u_{\la,\eps}+tRU_{\eps}|^q-|u_{\la,\eps}|^q-qtRu_{\la,\eps}^{q-1}U_\eps\dx.$$
For $L_3$, using Lemma \ref{q_ineq}-(iv) and the definition of $U_\eps$, we get
\begin{equation}
    L_3\geq
        \la\frac{(tR)^{q}}{q}\|U_\eps\|_{q}^{q}\geq\la\frac{(tR)^{q}}{q}\int_{B_r(y)}V_\eps^{q}\dx,\quad q\geq2.
\end{equation}
For $L_1$, we use Lemma \ref{regular}-(c) and \eqref{i-3} for the gradient term and \eqref{e7.2} for the nonlocal term to get
\begin{align*}
L_1&=C\left((tR)^2\int_{\Om}|\nabla u_{\la,\eps}|^{p-2}|\nabla U_{\eps}|^{2}\dx+(tR)^{p-1}\int_{\Om}|\nabla u_{\la,\eps}|\,|\nabla U_{\eps}|^{p-1}\dx\right)\\
    &\quad+C((tR)^2\eps[u_{\la,\eps}]_{s,p}^{p-2}[U_\eps]_{s,p}^2+(tR)^{p-1}\eps[u_{\la,\eps}]_{s,p}[U_\eps]_{s,p}^{p-1})\\
    &\leq C\left((tR)^2\int_{B_{2r}(y)}|\nabla U_\eps|^2\dx+(tR)^{p-1}\int_{B_{2r}(y)}|\nabla U_\eps|^{p-1}\dx\right)\\
    &\quad+C((tR)^2\eps[U_\eps]_{s,p}^2+(tR)^{p-1}\eps[U_\eps]_{s,p}^{p-1})\\
    &\leq C\left((tR)^2O\left(\eps^{\frac{2\alpha(N-p)}{p(p-1)}}\right)+(tR)^{p-1}O\left(\eps^{\alpha\frac{N-p}{p}}\right)\right)\\
    &\quad+C\bigg((tR)^2\left(O\left(\eps^{1+\alpha\frac{2(N-p)}{p(p-1)}}\right)+O\left(\eps^{1+\alpha(2-2s)}\right)\right)\\&\quad+(tR)^{p-1}\left(O\left(\eps^{1+\alpha\frac{N-p}{p}}\right)+O\left(\eps^{1+\alpha(p-1)(1-s)}\right)\right)\bigg).
\end{align*}
Notice that $1+\alpha (p-ps)>1+\alpha(p-1)(1-s)\geq1+\alpha(2-2s)$. We choose $\alpha>0$ small enough such that $\min\left\{1+\alpha(2-2s),1+\alpha\frac{2(N-p)}{p(p-1)}\right\}>\alpha\frac{N-p}{p}$. Hence,
\begin{equation}
    L_1\leq C\left((tR)^2O\left(\eps^{\frac{2\alpha(N-p)}{p(p-1)}}\right)+(tR)^{p-1}O\left(\eps^{\alpha\frac{N-p}{p}}\right)\right).
\end{equation}
 Combining \eqref{estimate-3}, \eqref{L2-estimate}, and the above inequalities, for any $i=1,2$,
 \begin{equation}\label{(d)}
     \begin{aligned}
         I_{\la,\eps}(w)&\leq I_{\la,\eps}(u_{\la,\eps})+\frac{(tR)^p}p K_1-\frac{(tR)^{\pst}}{\pst}K_2\\
         &\quad-\la\frac{(tR)^{q}}{q}\int_{B_r(y)}V_\eps^{q}\dx+\frac{(tR)^p}pO\left(\eps^{\alpha\frac{N-p}{p-1}}\right)\\&\quad+\frac{(tR)^{\pst}}{\pst}O\left(\eps^{\alpha \frac{N}{p-1}}\right)-C(tR)^{\pst-1} \eps^{\alpha \frac{N-p}p}\\
         &\quad+C(tR)^2O\left(\eps^{\frac{2\alpha(N-p)}{p(p-1)}}\right)+C(tR)^{p-1}O\left(\eps^{\alpha\frac{N-p}{p}}\right).
     \end{aligned}
 \end{equation}
For $t=1$, we choose $R_0>0$ such that for any $R\geq R_0$ and $\eps$ small enough,
\begin{equation}\label{(e)}
    \begin{aligned}
    &\left(\frac{R_0^p}p K_1-\frac{R_0^{\pst}}{\pst}K_2-R_0^{\pst-1}C\eps^{\alpha \frac{N-p}p}+CR_0^2O\left(\eps^{\frac{2\alpha(N-p)}{p(p-1)}}\right)+CR_0^{p-1}O\left(\eps^{\alpha\frac{N-p}{p}}\right)\right)\\&\quad+\left(\frac{R_0^p}p O\left(\eps^{\alpha\frac{N-p}{p-1}}\right)+\frac{R_0^{\pst}}{\pst}O\left(\eps^{\alpha \frac{N}{p-1}}\right)\right)<0.
\end{aligned}
\end{equation}
Thus $I_{\la,\eps}(u_{\la,\eps}+RU_\eps)<I_{\la,\eps}(u_{\la,\eps})$, for $R\geq R_0$, proving the first inequality of \eqref{energy estimate}. Next, we fix $R=R_0$ and define
\begin{equation}
\begin{aligned}
    \phi(t)&=\frac{(tR_0)^p}p K_1-\frac{(tR_0)^{\pst}}{\pst}K_2-\la\frac{(tR)^{q}}{q}\int_{B_r(y)}V_\eps^{q}\dx\\
    &\quad-(tR_0)^{\pst-1}C\eps^{\alpha \frac{N-p}p}+C(tR_0)^2O\left(\eps^{\frac{2\alpha(N-p)}{p(p-1)}}\right)+C(tR_0)^{p-1}O\left(\eps^{\alpha\frac{N-p}{p}}\right).
\end{aligned}
\end{equation}
For $t\in[0,1]$, we rewrite \eqref{(d)} as
\begin{equation}\label{(e2)}
    I_{\la,\eps}(u_{\la,\eps}+tR_0U_{\eps})\leq I_{\la,\eps}(u_{\la,\eps})+\phi(t)+o\left(\eps^{\al \frac{N-p}{p}}\right).
\end{equation}
As in the previous case, here we also have $\phi(0)=0$ and by \eqref{(e)}, $\phi(1)<0$. But because of the $q$ term, $\phi(t)<0$ for small $t>0$. So we consider two scenarios: (1) $\phi(t)\leq0$ for all $t\in[0,1]$ and (2) $\phi$ is positive for some $t\in(0,1)$. In the first scenario, for small enough $\eps$ we observe from \eqref{(e2)},
$$I_{\la,\eps}(u_{\la,\eps}+tR_0U_{\eps})<I_{\la,\eps}(u_{\la,\eps})+\frac1N  S_0^{\frac{N}{p}}.$$
Suppose the second scenario happens, then we can say that $\phi$ attains its positive maximum at some $t_\eps\in(0,1)$. Thus,
\begin{equation}\label{(f)}
    I_{\la,\eps}(u_{\la,\eps}+tR_0U_{\eps})\leq I_{\la,\eps}(u_{\la,\eps})+\phi(t_\eps)+o\left(\eps^{\alpha\frac{N-p}{p}}\right).
\end{equation}
Suppose $t_\eps\to0$ as $\eps\to0$. Then for small enough $\eps$, from \eqref{(f)} we get
$$I_{\la,\eps}(u_{\la,\eps}+tR_0U_{\eps})<I_{\la,\eps}(u_{\la,\eps})+\frac1N  S_0^{\frac{N}{p}}.$$
If $t_{\eps} \not \rightarrow 0$, then there exists some $T\in(0,1)$ such that $T<t_\eps<1$ for every $\eps$ small. Similar to the previous case,
\begin{equation}\label{fin_p_3}
    \begin{aligned}
        I_{\la,\eps}(u_{\la,\eps}+tR_0U_{\eps})&\leq I_{\la,\eps}(u_{\la,\eps})+\frac1N S_0^{\frac{N}{p}}-(TR_0)^{\pst-1}C\eps^{\alpha \frac{N-p}p}\\&\quad-\la\frac{(tR_0)^{q}}{q}\int_{B_r(y)}V_\eps^{q}\dx\\
        &\quad +CR_0^2O\left(\eps^{\frac{2\alpha(N-p)}{p(p-1)}}\right)+CR_0^{p-1}O\left(\eps^{\alpha\frac{N-p}{p}}\right)+o\left(\eps^{\alpha\frac{N-p}{p}}\right).
    \end{aligned}
\end{equation}
Now, we estimate the $q$ term. Using \eqref{lb-2}  in \eqref{fin_p_3}, we obtain \eqref{energy estimate-2} provided
\begin{align}\label{range-1}
    N-\frac{N-p}{p}q<\frac{2(N-p)}{p(p-1)}\Longleftrightarrow q>p^*-\frac{2}{p-1}.
\end{align}
This completes the proof.
\end{proof}
\begin{remark}
    Observe that for $2\leq p<3$, we have completely eliminated the term $L_3$ in the proof of Proposition \ref{energy_prop}. As a consequence, in this range, $\eps_\la$ becomes independent of $\la$, and \eqref{energy estimate}  and \eqref{energy estimate-2} hold for every $\eps \in (0, \eps_0)$.
\end{remark}
Next, we prove the $\text{(PS)}_c$ condition for the energy functional $I_{\la,\eps}$.
\begin{lemma}\label{PS cond}
    $I_{\la,\eps}$ satisfies $\text{(PS)}_c$ condition for every
    \begin{equation}
        c<I_{\la,\eps}(u_{\la,\eps})+\frac1NS_0^{\frac{N}{p}}.
    \end{equation}
\end{lemma}
\begin{proof}
    Let $\{u_n\}\subset X_0$ be a PS sequence of $I_{\la,\eps}$ at the level $c$. Then
    $I_{\la,\eps}(u_n)\to c$ and $\|I_{\la,\eps}'(u_n)\|_{X_0^\ast}\to0 \text{ as }n\to\infty.$ Observe that
    \begin{align*}
        c+C_1\rho_\eps(u_n)+o_n(1)&=I_{\la,\eps}(u_n)-\frac{1}{\pst}I_{\la,\eps}'(u_n)(u_n) \\&\geq \left(\frac1p-\frac1{\pst}\right)\rho_\eps(u_n)^p-\la C_q\rho_\eps(u_n)^q.
    \end{align*}
    Hence $\{u_n\}$ is bounded in $X_0$. By the ref{}lexivity of $X_0$, up to a subsequence, $u_n\rightharpoonup u_0$ in $X_0$. By Brezis-Lieb lemma,
    \begin{equation}\label{aba}
        \begin{aligned}
        &\|\nabla u_n\|_p^p-\|\nabla (u_n-u_0)\|_p^p=\|\nabla u_0\|_p^p+o_n(1),\\
        &[u_n]_{s,p}^p-[u_n-u_0]_{s,p}^p=[u_0]_{s,p}^p+o_n(1),\\
        &\|(u_n)_+\|_{\pst}^{\pst}-\|(u_n-u_0)_+\|_{\pst}^{\pst}=\|(u_0)_+\|_{\pst}^{\pst}+o_n(1).
    \end{aligned}
    \end{equation}
    Suppose $\{ u_n\}$ does not converge to $u_0$ i.e. $\rho_{\eps}(u_n-u_0)\geq C$ for all $n$.
    By \eqref{aba} and Lemma \ref{converg_limi},
    \begin{equation}\label{aba1}
        \rho_\eps(u_n-u_0)^p-\|(u_n-u_0)_+\|_{\pst}^{\pst}=I_{\la,\eps}'(u_n)(u_n-u_0)+o_n(1)=o_n(1).
    \end{equation}
    By the Sobolev inequality $S_0\norm{u^+}^p_{p^*} \le \rho_{\eps}(u)^p$ and \eqref{aba1}, we have
    \begin{align*}
        \|(u_n-u_0)_+\|_{\pst}^{\pst-p}=\frac{\|(u_n-u_0)_+\|_{\pst}^{\pst}}{\|(u_n-u_0)_+\|_{\pst}^p}&=\frac{\rho_\eps(u_n-u_0)^p+o_n(1)}{\|(u_n-u_0)_+\|_{\pst}^p}\\&\geq S_0+\frac{o_n(1)}{\rho_\eps(u_n-u_0)^p}.
    \end{align*}
    Since $\rho_{\eps}(u_n-u_0)\geq C$, $$\|(u_n-u_0)_+\|_{\pst}^{\pst-p}\geq S_0+o_n(1) \Longrightarrow \|(u_n-u_0)_+\|_{\pst}^{\pst}\geq S_0^{\frac{N}{p}}+o_n(1). $$
    Therefore, by \eqref{aba1}, for large $n$,
    \begin{align}
        \frac1NS_0^{\frac{N}{p}}&\leq \frac1N\|(u_n-u_0)_+\|_{\pst}^{\pst}+o_n(1)
        \no\\&=\frac1p\rho_\eps(u_n-u_0)^p-\frac1{\pst}\|(u_n-u_0)_+\|_{\pst}^{\pst}+o_n(1)\nonumber\\
        &=I_{\la,\eps}(u_n)-I_{\la,\eps}(u_0)+o_n(1)\nonumber\\
        &< c_{\la,\eps}+\frac1NS_0^{\frac{N}{p}}-I_{\la,\eps}(u_0).\label{parama1}
    \end{align}
    Next, we claim that $I_{\la, \eps}(u_0)\geq c_{\la,\eps}$. First, assuming the claim, we complete the proof. By \eqref{parama1}, $$\frac1NS_0^{\frac{N}{p}}< c_{\la,\eps}+\frac1NS_0^{\frac{N}{p}}-I_{\la,\eps}(u_0)\leq \frac1NS_0^{\frac{N}{p}},$$
    a contradiction. Hence, $u_n\to u_0$ in $X_0$ and $I_{\la,\eps}$ satisfies PS condition at the level $c$.

 Now we are left to prove the claim. Note that, if $I_{\la, \eps}(u_0)\geq0$, then the claim follows as  $c_{\la,\eps}<0$. Therefore, now we assume  $I_{\la, \eps}(u_0)<0$. Recall by Proposition \ref{B_r lemma}, $$I_{\la,\eps}\big|_{\pa B_{r_0}}>\delta_0>0.$$ For $t\geq0$, we define
    \begin{equation}\label{para-1}
        g(t)=I_{\la,\eps}\left(t\frac{u_0}{\rho_\eps(u_0)}\right)=\frac1pt^p-\frac1{\pst}t^{\pst}\frac{\|(u_0)_+\|_{\pst}^{\pst}}{\rho_{\eps}(u_0)^{\pst}}-\frac{\la}{q}t^{q}\frac{\|(u_0)_+\|_{q}^{q}}{\rho_{\eps}(u_0)^{q}}.
    \end{equation}
    Thus, $g(r_0)>\delta_0>0$. Thus, by continuity of $g$, $g(t)>0$ when $t$ lies in a neighbourhood of $r_0$. By \eqref{para-1},
    \begin{equation}
    \begin{aligned}
        g'(t)&=t^{p-1}-t^{\pst-1}\frac{\|(u_0)_+\|_{\pst}^{\pst}}{\rho_{\eps}(u_0)^{\pst}}-{\la}t^{q-1}\frac{\|(u_0)_+\|_{q}^{q}}{\rho_{\eps}(u_0)^{q}}\\
        &=t^{q-1}\left[t^{p-q}-t^{\pst-q}\frac{\|(u_0)_+\|_{\pst}^{\pst}}{\rho_{\eps}(u_0)^{\pst}}-{\la}\frac{\|(u_0)_+\|_{q}^{q}}{\rho_{\eps}(u_0)^{q}}\right].
    \end{aligned}
    \end{equation}
Since $I_{\la, \eps}(u_0)<0$, $u_0$ is nonzero. Further since $u_0$ is a critical point of $I_{\la,\eps}$,
\begin{align*}
g'(\rho_\eps(u_0))=\left< I_{\la,\eps}'(u_0),\frac{u_0}{\rho_\eps(u_0)}\right>=0.
\end{align*}
Thus $g$ has a positive critical point. But observe that $g$ is strictly decreasing near zero and $g(t)\to-\infty$ as $t\to\infty$. This implies that $g$ has at least two positive critical points. Suppose for contradiction, $g$ has more than two positive critical points then $h$ has three distinct zeros, where $h$ is defined as $$h(t):=t^{p-q}-t^{\pst-q}\frac{\|(u_0)_+\|_{\pst}^{\pst}}{\rho_{\eps}(u_0)^{\pst}}-{\la}\frac{\|(u_0)_+\|_{q}^{q}}{\rho_{\eps}(u_0)^{q}}.$$
But observe that $h$ has only one positive critical point, which is a contradiction by Rolle's theorem. Hence, $g$ has exactly two positive critical points $t_1<t_2$, where $t_1$ is the local minimum and $t_2$ is the global maximum. Since $g$ is positive for some $t$, $g(t_2)>0$. Therefore, $\rho_\eps(u_0)$ can not be $t_2$. Thus $\rho_\eps(u_0)=t_1$. Now suppose $\rho_\eps(u_0)\geq r_0$. Since $g$ is negative near $t=0$ and positive near $r_0$, $g$ has a critical point in $(0,r_0)$, which contradicts the fact that $g$ has exactly two critical points. Hence, $\rho_\eps(u_0)<r_0$. By Proposition \ref{B_r lemma}, we infer that $I_{\la,\eps}(u_0)\geq c_{\la,\eps}.$ This completes the proof of the lemma.
\end{proof}
Finally, we prove the existence of a second positive solution of \eqref{main_PDE} using the Mountain pass theorem.
\begin{proof}[{\bf Proof of Theorem \ref{second_sol}}]
Let $\la^\#$ and $\eps_\la$ be as in Proposition \ref{B_r lemma} and Proposition \ref{energy_prop} respectively. From Proposition \ref{B_r lemma}, we recall that $u_{\la,\eps}$ lies in the interior of $B_{r_0}$ and $\la^\#<\la_{**}$. Therefore, for any $0<\la<\la^\#$, $$I_{\la,\eps}\big|_{\pa B_{r_0}}\geq \frac1pr_0^p-C_1r_0^{\pst}-\la C_2r_0^q\geq 2\delta_0-\la C_2r_0^q>\delta_0>0.$$ By Proposition \ref{energy_prop}, There exists $R_0>0$ (enlarging $R_0$ if necessary) such that
\begin{align*}
&\rho_{\eps}(u_{\la,\eps}+R_0U_\eps)>r_0, \; I_{\la,\eps}(u_{\la,\eps}+R_0U_\eps)<I_{\la,\eps}(u_{\la,\eps})=c_{\la,\eps}<0, \text{ and }\\
    &I_{\la,\eps}(u_{\la,\eps}+tR_0U_\eps)<I_{\la,\eps}(u_{\la,\eps})+\frac1N S_0^{\frac{N}{p}},\quad\forall\,t\in[0,1].
\end{align*}
Define
\begin{align*}
&c:=\inf_{\gamma\in\Gamma}\max_{t\in[0,1]}I_{\la,\eps}(\gamma(t)), \\
&\text{ where } \Gamma:=\left\{\gamma\in C([0,1],X_0):\gamma(0)=u_{\la,\eps},\gamma(1)=u_{\la,\eps}+R_0U_\eps\right\}.
\end{align*}
Observe that the path $\gamma_0(t):=u_{\la,\eps}+tR_0U_\eps,\, t\in[0,1]$ lies in $\Gamma$. Thus, $$c\leq\max_{t\in[0,1]}I_{\la,\eps}(\gamma_0(t))<I_{\la,\eps}(u_{\la,\eps})+\frac1N S_0^{\frac{N}p}.$$
Therefore, by Lemma \ref{PS cond} and the Mountain pass theorem, there exists a critical point $v_{\la,\eps}$ of $I_{\la,\eps}$ such that $I_{\la,\eps}(v_{\la,\eps})=c$. Moreover, since $c\geq\delta_0>0>c_{\la,\eps}$, we get $v_{\la,\eps} \neq u_{\la,\eps}$.
\end{proof}
\begin{remark}\label{lambda-ind}
    If $u_{\la,\eps}\neq z_{\la,\eps}$ is the minimal solution of \eqref{main_PDE}, then we get two positive solutions of \eqref{main_PDE} with $z_{\la,\eps}\leq u_{\la,\eps}$ a.e. in $\Omega$. Suppose $u_{\la,\eps}=z_{\la,\eps}$ then, we again have two positive solutions of \eqref{main_PDE} and $u_{\la,\eps}\leq v_{\la,\eps}$.
\end{remark}

\appendix
\section{Existence of infinitely many nontrivial solutions}\label{appndx}
As discussed in the introduction, da Silva et al. \cite{DFB} have shown the existence of infinitely many nontrivial solutions of ($\tb{\mathcal{P}_{\la, 1}}$) using Krasnoselskii genus theory. Here, we show the existence of a sequence of nontrivial solutions of \eqref{main_PDE} with negative energy whose energy converges to zero using the Dual Fountain Theorem. Consider the following energy functional associated with \eqref{main_PDE}:
\begin{align*}
    \Tilde{I}_{\la,\eps}(u) := \frac1p\rho_\eps(u)^p-\frac{\lambda}{q}\|u\|_q^q-\frac{1}{\pst}\|u\|_{\pst}^{\pst}, \; \forall \, u \in X_0.
\end{align*}
Observe that $\Tilde{I}_{\la,\eps} \in C^1(X_0, \R)$ and every critical point of $\Tilde{I}_{\la,\eps}$ corresponds to a weak solution of \eqref{main_PDE}. Since $X_0$ and $X_0^\ast$ are separable Banach spaces, we select a shrinking and Markushevich basis for $X_0$ (see \cite[Theorem 1.22]{HMVZ}), i.e. there exists $\{\phi_n\}_{n\in\NN}$ in $X_0$ and $\{\phi_n^\ast\}_{n\in\NN}$ in $X_0^\ast$ such that $X_0=\ov{\{\phi_n:n\in\NN\}}$, $X_0^\ast=\ov{\{\phi_n^\ast:n\in\NN\}}$ and $\langle\phi_n^\ast,\phi_m\rangle=\delta_{n,m}$, $n,m\in \NN$, where $\delta_{n,m}$ is the Kronecker Delta function. We define
$$X_j:=\R \phi_j,\quad Y_k:=\bigoplus_{j=1}^kX_j,\quad Z_k:= \overline{\bigoplus_{j=k}^{\infty}X_j},$$
 and consider the antipodal action of $G=\ZZ/2$ on $X_0=\ov{\bigoplus_{j\geq1}X_j}$. Note that $G$ acts isometrically on $X_0$, the spaces $X_j$ are invariant, $X_j\simeq \R,\,j\in\NN$ and the antipodal action of $G$ on $\R$ is admissible (see \cite[Example 3.3, Theorem D.17]{W}). Now, we state the following theorem (see \cite[Theorem 3.18]{W}).

\begin{theorem}[Dual Fountain Theorem]\label{dual_fount_theorem}
    Let $\phi\in C^1(X_0,\R)$ be an invariant functional, i.e. $\phi\circ g=\phi$ for all $g\in G$. If for every $k\geq k_0$, there exists $0<r_k<\rho_k$ such that
    \begin{enumerate}
        \item[\rm{(a)}] $a_k:=\inf_{u\in Z_k,\rho_\eps(u)=\rho_k}\phi(u)\geq0$,
        \item[\rm{(b)}] $b_k:=\max_{u\in Y_k,\rho_\eps(u)=r_k}\phi(u)<0$,
        \item[\rm{(c)}] $d_k:=\inf_{u\in Z_k,\rho_\eps(u)\leq\rho_k}\phi(u)\rightarrow 0$ as $k\to\infty$,
        \item[\rm{(d)}] Every sequence $u_{r_j}\in X_0$ satisfying $$u_{r_j}\in Y_{r_j}, \; \phi(u_{r_j})\to c\in[d_k,0), \; \left(\phi\mid_{Y_{r_j}}\right)'(u_{r_j})\to 0, \; \text{ as }r_j\to\infty,$$
        has a subsequence converging to a critical point of $\phi$.
    \end{enumerate}
    Then, $\phi$ has a sequence of negative critical values converging to $0$.
\end{theorem}

\begin{lemma}\label{lemma_betak}
    For $q \in (1,p)$, define $$\beta_k:=\sup_{u\in Z_k, \rho_\eps(u)=1}\|u\|_{q}.$$ Then $\beta_k \ra 0$ as $k \ra \infty$.
\end{lemma}
\begin{proof}
    Observe that $\{ \beta_k \}$ is a decreasing sequence and hence $\beta_k\to\beta$ for some $\beta  \ge 0$. By the definition of supremum, for every $k\geq0$, we choose $u_k\in Z_k$ such that $\rho_\eps(u_k)=1$ and $\|u_k\|_q>\frac{\beta_k}{2}$. Let $T\in X_0^\ast$. Since $\{\phi_n^\ast\}$ forms a basis of $X_0^\ast$, for some sequence $\{a_n\}\subset\R$, we write $T=\sum_{n\geq1}a_n\phi_n^\ast$. First, suppose $T$ is a finite sum i.e. $T=\sum_{n=1}^{n_1}a_n\phi_n^\ast$. Since $Z_k=\ov{\text{span}\{\phi_k,\phi_{k+1},\ldots\}}$ then for all $k\geq n_1+1$, $T(u_k)=\sum_{n=1}^{n_1}a_n\phi_n^\ast(u_k)=0$ as $\langle\phi_n^\ast,\phi_m\rangle=\delta_{n,m}$. Now if it is not a finite sum then for $\eps>0$ choose $n_0>0$ such that $\left|T(u_k)-\sum_{n=1}^{n_0}a_n\phi_n^\ast(u_k)\right|<\eps$. Now for large $k$, $$|T(u_k)|\leq\left|T(u_k)-\sum_{n=1}^{n_0}a_n\phi_n^\ast(u_k)\right|+\sum_{n=1}^{n_0}|a_n|\,|\phi_n^\ast(u_k)|<\eps.$$Thus $T(u_k)\to0$ as $k\to\infty$ for all $T\in X_0^\ast$. That means $u_k\rightharpoonup0$ in $X_0$. The Sobolev embedding theorem implies that $u_k\to0$ in $L^q(\Omega)$. Thus
    \begin{align*}
      \beta = \lim_{k \ra \infty} \beta_k \le \lim_{k \ra \infty} 2 \norm{u_k}_q = 0,
    \end{align*}
    as required.
\end{proof}

The following proposition states that \eqref{main_PDE} admits infinitely many nontrivial solutions with negative energy, and this sequence of energy converges to zero.
\begin{proposition}\label{inf_non_triv}
    Let $q\in(1,p)$ and $\eps\in(0,1]$. Then there exists $\lambda^\ast>0$ such that for all $\lambda\in(0,\lambda^\ast)$, \eqref{main_PDE} has a sequence of nontrivial solutions $\{u_n\}$ such that the following hold:
    \begin{align*}
        \Tilde{I}_{\la,\eps}(u_n) < 0, \; \forall \, n \in \mathbb{N} \, \text{  and  } \, \Tilde{I}_{\la,\eps}(u_n) \ra 0, \text{ as } n \ra \infty.
    \end{align*}
\end{proposition}

\begin{proof}
 Since $\Tilde{I}_{\la,\eps}$ is an even function, i.e., $\Tilde{I}_{\la,\eps}(u)=\Tilde{I}_{\la,\eps}(-u)$, it is enough to check the conditions (a)-(d) in Theorem \ref{dual_fount_theorem}.
  Let
  $$R:=\left(\frac{\pst\SS_0^{\frac{\pst}{p}}}{2p}\right)^{\frac{1}{\pst-p}}.$$
  Then for any $u\in Z_k$ with $\rho_\eps(u)\leq R$,
\begin{align}
    \Tilde{I}_{\la,\eps}(u)&=\frac1p\rho_\eps(u)^p-\frac{\lambda}{q}\|u\|_q^q-\frac{1}{\pst}\|u\|_{\pst}^{\pst}
    \no\\&\geq \frac1p\rho_\eps(u)^p-\frac{\lambda}{q}\beta_k^q\rho_\eps(u)^q-\frac{1}{\pst\SS_0^{\pst/p}}\rho_\eps(u)^{\pst}\nonumber\\
    &\geq \frac{1}{2p}\rho_\eps(u)^p-\frac{\lambda}{q}\beta_k^q\rho_\eps(u)^q.\label{e5}
\end{align}
We choose
$\rho_k:=\left(\frac{2p\lambda\beta_k^q}{q}\right)^{\frac{1}{p-q}}. $
Hence for large $k$, $u\in Z_k$ and $\rho_\eps(u)=\rho_k$, $\Tilde{I}_{\la,\eps}(u)\geq0$, i.e., (a) holds. Observe that $Y_k$ are finite-dimensional spaces. Hence, all norms on $Y_k$ are equivalent. Since $\lambda>0$, we  choose a small enough $r_k<\rho_k$ such that (b) holds. For $k$ large, $u\in Z_k$ and $\rho_\eps(u)\leq\rho_k$, it is easy to see that $$\frac{-\lambda}{q}\beta_k^q\rho_k^q\leq d_k\leq \frac1p\rho_k^p.$$ In view of Lemma \ref{lemma_betak} and using $q<p$, $\rho_k\to0$ as $k \ra \infty$. Thus, we get (c). Recall that by \cite[Lemma 2.4(ii)]{DFB}, $\Tilde{I}_{\la,\eps}$ satisfies (PS)$_c$ for any $c$ satisfying \eqref{PSc}. Therefore, for $q<p$ and for every $\lambda\in(0,\lambda^\ast)$, where $\la^\ast$ is defined in \eqref{e5.01}, it holds
$$ \frac1N S_0^{\frac{N}{p}}>|\Omega|\left(\frac1p-\frac1{\pst}\right)^{-\frac{q}{\pst-q}}\left(\lambda\left(\frac1q-\frac1p\right)\right)^{\frac{\pst}{\pst-q}}.$$
Hence, (d) holds. Therefore, we conclude the proof by applying Theorem \ref{dual_fount_theorem}.
\end{proof}
\section*{Acknowledgments}
The research of M.~Bhakta is partially supported by the DST Swarnajaynti Fellowship (SB/SJF/2021-22/09). The research of N.~Biswas is partially supported by the SERB National Postdoctoral Fellowship (PDF/2023/000038). The research of P.~Das is partially supported by the NBHM Fellowship (0203/5(38)/ 2024-R\&D-II/11224).









\bibliographystyle{abbrvnat}

\end{document}